\documentclass{siamart250211}
\usepackage[protrusion=true,expansion=true,final,tracking=true,kerning=true,spacing=true]{microtype}
\microtypecontext{spacing=nonfrench}
\UseMicrotypeSet[protrusion]{basicmath}
\usepackage{mathrsfs}
\usepackage{braket}

\usepackage{lipsum}
\usepackage{amsfonts}
\usepackage{graphicx}
\usepackage{epstopdf}
\usepackage{dsfont}
\usepackage{amssymb}
\usepackage[sort, square, numbers]{natbib}

\usepackage{algorithm} %
\usepackage{algpseudocode} %

\ifpdf
  \DeclareGraphicsExtensions{.eps,.pdf,.png,.jpg}
\else
  \DeclareGraphicsExtensions{.eps}
\fi

\crefformat{equation}{(#2#1#3)}
\Crefformat{equation}{(#2#1#3)}

\newsiamremark{remark}{Remark}
\newsiamremark{hypothesis}{Hypothesis}
\crefname{hypothesis}{Hypothesis}{Hypotheses}
\newsiamthm{claim}{Claim}
\newsiamremark{problem}{Problem}

\newsiamremark{example}{Example}

\headers{Average-case exact regularization of LPs}{M. P. Friedlander, S. Kubal, Y. Plan, and M. S. Scott} 

\title{%
   Average-case thresholds for\\ exact regularization of linear programs%
   \thanks{%
      Authors listed in alphabetical order. SK and MSS led
      theoretical development and manuscript preparation with guidance
      and contributions from MPF and YP.
      Research funded in part by the Natural Sciences and Engineering Council of Canada.
   }%
}

\author{Michael P. Friedlander\footnotemark[2]  \footnotemark[3]
\and Sharvaj Kubal\footnotemark[2]
\and Yaniv~Plan\footnotemark[2]
\and \mbox{Matthew S. Scott}\thanks{Dept.\@ of Mathematics; $^\ddagger$Dept.\@ of Computer Science, University of British Columbia, Canada.}
}

\usepackage{amsopn}

\makeatletter
\newcommand*{\addFileDependency}[1]{%
  \typeout{(#1)}%
  \@addtofilelist{#1}%
  \IfFileExists{#1}{}{\typeout{No file #1.}}%
}
\makeatother

\input{siam_preamble.sty}
\usepackage{enumitem}
\usepackage{mathtools}

\ifpdf
\hypersetup{
  pdftitle={Average-case thresholds for exact regularization of linear programs},
  pdfauthor={M.P. Friedlander, S. Kubal, Y. Plan, M.S. Scott}
}
\fi

\begin{document}

\maketitle
\begin{abstract}
Small regularizers can preserve linear programming solutions exactly. This paper provides the first average-case analysis of exact regularization: with a standard Gaussian cost vector and fixed constraint set, bounds are established for the probability that exact regularization succeeds as a function of regularization strength. Failure is characterized via the Gaussian measure of inner cones, controlled by novel two-sided bounds on the measure of shifted cones. Results reveal dimension-dependent scaling laws and connect exact regularization of linear programs to their polyhedral geometry via the normal fan and the Gaussian (solid-angle) measure of its cones. Computable bounds are obtained in several canonical settings, including regularized optimal transport. Numerical experiments corroborate the predicted scalings and thresholds.

\end{abstract}



\section{Introduction}
\label{loc:body.introduction}

Regularization transforms ill-posed optimization problems into well-conditioned ones by adding penalty terms. These penalties provide uniqueness, stability under data perturbations, and algorithms with improved convergence. For linear programs, a remarkable phenomenon occurs: for sufficiently small regularization, solutions of the regularized problem remain solutions of the unregularized problem. The exact regularization threshold, however, depends on problem geometry and cannot be determined without first solving the original problem. We develop an average-case analysis that treats the cost vector as a Gaussian random variable and establish tail bounds describing how the exact regularization threshold concentrates near dimension-dependent values. This analysis reveals that exact regularization depends on the Gaussian measure of shifted normal cones and the geometry of polytope vertices.
The bounds are computable and furnish explicit estimates for problems such as quadratically regularized optimal transport and linear regularization on the $\infty$-norm ball.

The linear program with data $(A, b, g)\in \mathbb{R}^{m \times n} \times \mathbb{R}^m \times \mathbb{R}^n$ takes the form
\begin{equation} \label{loc:optimization_problem_p0_exact_regularization.statement}
\mathrm{minimize} \set{-g \cdot x | Ax \leq b}.
\tag{\ensuremath{P_{0}}}
\end{equation}
Its regularized counterpart
\begin{equation}
\label{loc:optimization_problem_pep_of_exact_regularization.statement}
\mathrm{minimize} \set{-g \cdot x + \varepsilon\psi(x) | Ax \leq b},
\tag{\ensuremath{P_{\varepsilon}}}
\end{equation}
retains the same constraint set but augments the objective with the penalty $\varepsilon\psi(x)$, where $\psi:\mathbb{R}^n \to (-\infty, \infty]$ is a proper convex function and $\varepsilon \geq 0$.

For sufficiently small regularization parameter $\varepsilon$, every solution of the regularized problem $P_\varepsilon$ solves the original problem $P_0$. That is, $\mathrm{Sol}(P_\varepsilon) \subseteq \mathrm{Sol}(P_0)$. We call this property \emph{exact regularization}. Deterministic analysis of exact regularization began with \citet{mangasarian_NonlinearPerturbationLinear_1979}, whose work and subsequent contributions \cite{mangasarian_NormalSolutionsLinear_1984a, mangasarian_RegularizedLinearPrograms_1999} establish conditions under which $P_{0}$ and $P_{\varepsilon}$ share solutions, including cases where only local optimality is preserved.

Exact regularization addresses both computational and theoretical challenges in linear programming. Computationally, LPs may exhibit nonunique solutions and extreme sensitivity to data perturbations. Regularization with a strongly convex function $\psi$ (such as $\frac{1}{2} \|\cdot\|^2$) attenuates these issues. The dual relationship between strong convexity and smoothness produces a smooth dual problem, which admits first-order algorithms with nonasymptotic linear convergence rates~\cite{mangasarian_IterativeSolutionLinear_1981, yin_AnalysisGeneralizationsLinearized_2010, lai_AugmentedL1NuclearNorm_2013}. Theoretically, perturbation analysis reveals fundamental properties of LPs: \citet{mangasarian_UniquenessSolutionLinear_1979} leverages perturbations to characterize solution uniqueness, while \citet{mangasarian_NormalSolutionsLinear_1984a} connected solvability to projections onto feasible regions. This framework extends naturally beyond linear programming. \citet{ferris_FinitePerturbationConvex_1991}, \citet{friedlander_ExactRegularizationConvex_2008}, and \citet{deng_ExactRegularizationIts_2015} develop exact regularization theory for general convex programs and establish connections with exact penalization \cite{bertsekas_NecessarySufficientConditions_1975}. Most recently, \citet{gonzalez-sanzQuantitativeConvergenceQuadratically2025} apply these ideas to quadratically regularized optimal transport.

Exact regularization theory characterizes the threshold value
\begin{equation}
\label{eq:eps_threshold}
\overline{\varepsilon}:= \sup \set{\varepsilon \ge 0 \mid \mathrm{Sol}(P_{\varepsilon'})\subseteq \mathrm{Sol}(P_{0})\,\text{for all }\varepsilon' \in [0,\varepsilon)}
\end{equation}
through the dual of a selection problem
\begin{equation*}
\minimize_{x} \set{ \psi(x) | Ax \leq b, \ -g\cdot x \leq p^\star }.
\tag{\ensuremath{P^\psi}}
\end{equation*}
Here $p^\star = \inf \{ -g\cdot x \mid Ax \leq b \}$ is the optimal value of $P_0$. \citet[Theorem 1]{mangasarian_NonlinearPerturbationLinear_1979} establishes that $\overline{\varepsilon}$ equals the inverse of the Lagrange multiplier of the objective constraint in $P^\psi$.
 \citet[Theorem 2.1, Corollary 2.2]{friedlander_ExactRegularizationConvex_2008} characterize the existence of this multiplier as both necessary and sufficient for exact regularization in general convex programs. Determining $\overline{\varepsilon}$ thus requires solving both $P_0$ and $P^\psi$. Without prior knowledge of $\overline{\varepsilon}$ from the LP data $(g, A, b, \psi)$, we adopt a probabilistic approach. Randomizing problem data enables average-case analysis, a methodology well-established in computational complexity \cite{johnson_NPcompletenessColumnOngoing_1984} fruitful for convex optimization \cite{wendel_ProblemGeometricProbability_1962, cheung_SolvingLinearPrograms_2004, amelunxen_LivingEdgePhase_2014}. 

We randomize the cost vector of $P_{0}$ as $g \sim \mathcal{N}(0, I_{n})$, a standard Gaussian vector in $\mathbb{R}^n$. With $(A, b, \psi)$ fixed, we quantify the probability that $\mathrm{Sol}(P_{\varepsilon})\subseteq \mathrm{Sol}(P_0)$ for some $\varepsilon$. The choice of a Gaussian distribution provides rotational symmetry, treating all directions equally. For large $n$, the Gaussian random vector $g$ concentrates around the sphere $\sqrt{n}\,\mathbb{S}^{n-1}$ and can be viewed as a random direction with approximately fixed norm~\cite[Section 3.3.3]{vershynin_HighDimensionalProbabilityIntroduction_2018}. 

\subsection{Polyhedral geometry and normal cone structure} \label{loc:introduction_to_exact_regularisation_with_proof.mathematical_set:up_and_the_central_theorem}

We require the following geometric framework. Denote the nonempty polyhedral feasible region of $P_{0}$ by $Q := \set{ x \in \mathbb{R}^n | Ax \leq b}$. We focus on polytopes (bounded polyhedra) $Q$, which ensures finite optimal values $p^*$. For unbounded $Q$, our results apply conditionally on the existence of solutions.

The normal cone to $Q$ at $x\in Q$ is defined by 
\begin{equation*}
N(x) := \{ v \in \mathbb{R}^n \mid v\cdot (y-x) \leq 0\,\text{ for all } y \in Q \}.
\end{equation*}
This cone admits the finite representation
$N(x) = \conv\cone \set{ a_{i} | i \in J(x) }$,
where $a_i$ denotes the $i$th row of $A$ and $J(x):= \set{i \in[m] | a_{i}\cdot x= b_{i} }$ indexes the active constraints at $x$~\cite[Section A.5.2]{hiriart-urruty_FundamentalsConvexAnalysis_2001}.

A vertex is a singleton exposed face of $Q$ \cite[Section A.2.3]{hiriart-urruty_FundamentalsConvexAnalysis_2001}:
\begin{equation*}
\vrt(Q) := \{ z \in Q \mid Q \cap H = \{ z \}\,\text{ for a supporting hyperplane }H \text{ of }Q\}.
\end{equation*}
For polyhedral $Q$, the finite set $\vrt(Q)$ coincides with the extreme points $\mathrm{ext}(Q)$. The point $x$ is a vertex if and only if $\mathrm{rank}(A_{J(x)})= n$, where $A_{J(x)}$ denotes the submatrix of $A$ with rows indexed by $J(x)$. Vertices with exactly $n$ active constraints ($|J(x)|=n$) are \emph{nondegenerate}; those with $|J(x)|>n$ are \emph{degenerate}~\cite{bertsimas1997introduction}. Our analysis centers on the normal cones $N(z)$ at vertices $z\in\vrt(Q)$.

For a set $A \subseteq \mathbb{R}^n$ and $w  \in\mathbb{R}^n$, denote the distance between $w$ and $A$ by $\textrm{dist}(w, A) := \inf_{a \in A} \lVert w-a\rVert$. For a \emph{closed convex} set $A \subseteq \mathbb{R}^n$, the projection of $w$ onto $A$ is unique and given by $\proj_A w := \argmin_{a  \in A} \lVert w-a\rVert$. Given a cone $V \subseteq \mathbb{R}^n$, define its \emph{dual} $V^* = \set{ y | y \cdot x \geq 0 \text{ for all } x  \in V }$.

\subsection{Optimality and Gaussian measure}

A point $z \in Q$ is optimal for $P_{0}$ if and only if
\begin{equation}
\label{eq:optimality_p0}
  g \in N(z).
\end{equation}
Since the cost vector $g$ is Gaussian, the probability that $z$ solves $P_0$ equals the Gaussian measure of the normal cone $N(z)$. The standard Gaussian measure of any measurable subset $B$ of $\mathbb{R}^n$ is
\begin{equation*}
\gamma(B) := \mathbb{P}(g \in B) = \int_{B} {(2\pi)^{-n/2}} e^{ -\|x\|^2/2 }\,dx.
\end{equation*}
When $B$ is a cone, its Gaussian measure equals its normalized solid angle~\cite[Equation 1]{ribando_MeasuringSolidAngles_2006}, which is the proportion of the sphere captured by the cone. Thus~\eqref{eq:optimality_p0} conveys the probabilistic statement that $z \in Q$ is a solution to $P_{0}$ with probability $\gamma(N(z))$. Since $\gamma$ assigns zero measure to lower-dimensional sets, only full-dimensional normal cones matter. The following proposition shows that these occur precisely at vertices of $Q$, where they partition $\mathbb{R}^n$ when $Q$ is a polytope.

\begin{proposition}[Normal cone properties]
\label{loc:normal_cones_and_faces.statement}
Let $Q\subset \mathbb{R}^n$ be polyhedral.
\begin{enumerate}
\item The normal cone $N(z)$ at $z \in Q$ is $n$-dimensional if and only if $z \in \vrt(Q)$.
\item If $x \in Q$ and $z \in \vrt(Q)$ are distinct, then $N(x) \cap N(z)$ is a cone of dimension strictly less than $n$. Furthermore, $N(x) \cap \mathrm{int}\,N(z) = \emptyset$.
\item If $Q$ is additionally a polytope, then $\cup_{z \in \vrt(Q)}N(z) = \mathbb{R}^n$.
\end{enumerate}
\end{proposition}

This near-partitioning of the space by vertex normal cones forms the \emph{normal fan} of $Q$~\cite{lu_NormalFansPolyhedral_2008}.
The following proposition is a consequence of \eqref{eq:optimality_p0} and \Cref{loc:normal_cones_and_faces.statement}.

\begin{proposition}[Uniqueness in $P_{0}$]
\label{loc:uniqueness_of_solution_in_a_random_linear_program.statement}
If the feasible region $Q$ is a polytope, then $P_{0}$ has a unique solution at a vertex of $Q$ with probability 1. This (almost surely defined) unique solution $z^\star\in \vrt(Q)$ has distribution
\begin{equation*}
\mathbb{P}\{ z^\star = z \} = \gamma(N(z)) \qquad (z \in \vrt(Q)).
\end{equation*}
\end{proposition}

The proofs of these propositions appear in \Cref{sec:aux_proofs}.

\subsubsection{Regularization and shifted cones}

To analyze the regularized program $P_{\varepsilon}$, we first establish its optimality conditions, adopting the convention that $(\varepsilon\psi) = \delta_{\dom\psi}$ when $\varepsilon=0$ (scalar multiplication otherwise). This convention
implies that multiplying $\psi$ by zero preserves the effective domain
of $\psi$, and thus $P_\varepsilon$ is different
from $P_0$, even when $\varepsilon = 0$, at least in the case where
$\dom(\psi) \neq \mathbb{R}^n$. This convention appears in \citet[Lemma 4.3]{aravkin_VariationalPropertiesValue_2013}, and \citet[pg.~141]{zalinescu_ConvexAnalysisGeneral_2002}.

A point $x \in Q$ is optimal for $P_{\varepsilon}$ if and only if
\begin{align}
    g  \in N(x) + \partial (\varepsilon\psi)(x),
    \label{eq:optimality_peps}
\end{align}
under the regularity condition~\cite[Theorem 23.8]{rockafellar_ConvexAnalysis_1970}
\begin{equation}
\label{eq:regularity_relint}
\mathrm{relint}(\mathrm{dom}\,\psi) \cap Q \neq \emptyset. \tag{R}
\end{equation}
Note that the optimality condition~\eqref{eq:optimality_peps} is sufficient for optimality even without~\eqref{eq:regularity_relint}, and consequently our lower bounds on the probability of exact regularization continue to hold without regularity.

The optimality condition~\eqref{eq:optimality_peps} for $P_{\varepsilon}$ suggests that the Gaussian measure of \emph{shifted} cones should play a central role in the analysis. The following theorem (proved in \Cref{loc:body.proofs}) quantifies how shifting a cone by a vector $w$ changes its Gaussian measure. The effect depends on the alignment of $w$ with the cone and its dual. 
\begin{theorem}[Gaussian measure of shifted cones]
\label{loc:gaussian_measure_of_shifted_cones.statement}
For any cone $V  \subseteq \mathbb{R}^{d}$ with dual cone $V^*$, and any $w \in \mathbb{R}^d$, the Gaussian measure of the shifted cone $V+w$ satisfies
\begin{align*}
\gamma(V+w) &\in \left[\, \gamma(V) \exp\left( - \tfrac{1}{2}\|w\|^2 - \mathrm{dist}(w, -V^*)\sqrt{d} \right), \right. \\
&\quad \left.\ \, \gamma(V) \exp \left( - \tfrac{1}{2} \|\proj_{V^*} w\|^2 + \mathrm{dist}(w, V^*)\sqrt{ d } \right) \,\right].
\end{align*}
\end{theorem}

These bounds immediately yield the following simplified estimates.
\begin{corollary}[Simpler measure of shifted cones]
\label{loc:simpler_measure_of_shifted_cones.statement}
For any cone $V  \subseteq \mathbb{R}^{d}$ and any $w \in \mathbb{R}^d$,
\begin{align*}
\gamma(V+w) &\in \left[\, \gamma(V) \exp\left( -\tfrac{1}{2}\norm{w}^2 - \norm{w}\sqrt{d} \right), \right. \\
&\quad \left.\ \, \gamma(V) \exp \left(\norm{w}\sqrt{ d } \right) \,\right].
\end{align*}
\end{corollary}

\subsubsection{Exact regularization phenomenon}

We define the \emph{exact regularization event} $\ER(\varepsilon)$ as the event where the unique solution $z^\star$ of $P_0$ also solves $P_\varepsilon$.
The implications of this event extend beyond the specific parameter value. It follows from the arguments of \citet[Theorem 2.1(d)]{friedlander_ExactRegularizationConvex_2008} that whenever $\ER(\varepsilon)$ occurs, $z^*$ uniquely solves $P_{\varepsilon'}$ for all $\varepsilon'\in [0, \varepsilon)$.
Consequently, the event $\ER(\varepsilon)$ implies that \emph{exact regularization} satisfies $\overline{\varepsilon} \geq \varepsilon$ as per~\eqref{eq:eps_threshold}.

The \emph{inner cone} and \emph{margin} sets, respectively, characterize the success and failure of exact regularization conditional on $\{ z^\star =z \}$ for a vertex $z\in\vrt(Q)$:
\begin{align}\label{eq:inner_cone_grad}
g &\in N(z) \cap [N(z) + \partial(\varepsilon\psi)(z)]\quad \text{(inner cone;  $\mathrm{ER}(\varepsilon)$ succeeds)},\\
g &\in N(z) \setminus [N(z) + \partial(\varepsilon\psi)(z)]\quad \text{(margin set; $\mathrm{ER}(\varepsilon)$ fails)}.\label{eq:margin}
\end{align}
The inner cone represents vectors in $N(z)$ that remain in the normal cone after regularization. On the other hand, the \emph{margin} set captures vectors in $N(z)$ that exit the normal cone under regularization
(see \cref{fig:pubfig:250713:membership:representer.pdf}). 

To compute the probability of exact regularization, observe that
\begin{equation}\label{eq:pr-exact}
\begin{aligned}
\prob(\ER(\varepsilon)) &=\textstyle \sum_{z  \in \vrt(Q)} \prob(\{z^\star=z\} \text{ and }
\ER(\varepsilon))\\
&= \textstyle \sum_{z  \in \vrt(Q)} \prob(g  \in N(z)  \cap [N(z) + \partial(\varepsilon \psi)(z)])\\
&= \textstyle \sum_{z  \in \vrt(Q)} \gamma(N(z)  \cap [N(z) + \partial(\varepsilon \psi)(z)]).
\end{aligned}
\end{equation}
This is formalized in the following theorem.

\begin{theorem}[Probability and normal cone geometry]
\label{loc:geometric_probabilistic_results.statement}
If $Q$ is a polytope and $\psi:\mathbb{R}^n \to (-\infty, \infty]$ is a proper, convex function that satisfies \eqref{eq:regularity_relint}, then~\eqref{eq:pr-exact} holds.
\end{theorem}

Note that quantifying the probability of $\ER(\varepsilon)$ in terms of $\varepsilon$ gives, using standard arguments, the expectation of the regularization threshold $\overline{\varepsilon}$:
\begin{align}
    \mathbb{E}\,\overline{\varepsilon} =\textstyle \int_{0}^\infty \mathbb{P}(\ER(\varepsilon))\,d\varepsilon.
    \label{eq:expectation_bound}
\end{align}
Because both $N(z)$ and $N(z)+ \partial(\varepsilon \psi)(z)$ are closed convex sets (arising as subdifferentials of convex functions), their interesection is indeed Borel measureable and thus $\ER(\varepsilon)$ is a valid event in a probability space.

\subsubsection{Extensions and remarks}

\Cref{loc:geometric_probabilistic_results.statement} allows for $\partial (\varepsilon \psi)(x) = \emptyset$ at vertices $x$ of $Q$. The Minkowski sum is then empty, and the Gaussian measure of the empty set is zero.
This occurs when $\psi$ has no affine minorant at $z^\star$.
For any $\varepsilon > 0$, the solution of $P_\varepsilon$
cannot match that of $P_0$, and the probability of exact regularization is zero. 
An empty subdifferential can result from a discontinuity of $\psi$ at $z^\star$. 
An example of an empty subdifferential is that of the entropic regularization on the simplex. Indeed, the (negative) entropy function $\psi(x)= \sum_{j=1}^n x_j \log x_j$ has unbounded gradients when approaching any of the vertices of the simplex. \Cref{loc:geometric_probabilistic_results.statement} confirms that exact regularization never occurs for entropic regularization on the simplex.

Beyond exact regularization, our theory admits two alternative interpretations. First, taking $\psi(x) = p \cdot x$ yields the linear program $P_\varepsilon$ with cost vector $-g + \varepsilon p$. Our results then provide probabilistic guarantees for the robustness of $P_0$ under cost vector perturbations, which in effect characterizes properties of the dual solutions. Second, taking $\psi$ with an effective domain strictly contained in $Q$ models a constraint or barrier function. Vertices outside $\dom(\psi)$ cannot be exactly regularized for any $\varepsilon \geq 0$. The exact regularization probability becomes the probability that solutions to a random linear program $P_0$ satisfy the constraint imposed by $\psi$.

\subsection{Summary of contributions}
\label{loc:introduction_to_exact_regularisation_with_proof.summary_of_contributions}

Our focus for the remainder of this work is to find good estimates on the Gaussian
measure of polyhedral inner cones $N(z)  \cap [N(z) + \partial (\varepsilon \psi)(z)]$ at vertices $z$.

\begin{enumerate}[leftmargin=*,label=\arabic*.,itemsep=0.2\baselineskip,topsep=0.25\baselineskip]

\item \emph{Gaussian measures of shifted cones.} In \Cref{loc:gaussian_measure_of_shifted_cones.statement}, we establish two-sided bounds on how the Gaussian measure of a cone changes under a shift. This result underlies our entire probabilistic analysis of exact regularization.

\item \emph{Geometric characterization}. In \Cref{loc:geometric_probabilistic_results.statement}, we show that the probability of exact regularization equals the sum of Gaussian measures of inner cones $N(z) \cap [N(z) + \partial(\varepsilon\psi)(z)]$ across all vertices of $Q$.

\item \emph{Inner-cone bounds}. We develop a comprehensive suite of bounds on the Gaussian measure of inner cones through membership conditions (\Cref{loc:gaussian_measure_of_the_margin_resulting_from_an_inward_shift.statement}, \Cref{loc:corollary_of_the_inward_shift_result.statement}) and generalize these results via representer vectors (\Cref{loc:bound_via_representer_vectors.statement}).

\item \emph{Margin sets}. \Cref{loc:body.main_results.using_the_boundary_of_normal_cones} relaxes the membership condition. \Cref{loc:gaussian_measure_of_the_margin.statement} provides two-sided bounds on the Gaussian measure of margin sets $T \setminus [T+w]$ for polyhedral cones $T$ via facet structure. Then follows \Cref{loc:applied_margin_measure_for_exact_regularization.statement}, a general bound
on the probability of exact regularization, and \Cref{loc:linear_perturbed_polytope_regularization.statement}, a simpler bound specialized to linear regularization.

\item \emph{Applications to $\Binf$}. \Cref{loc:body.b_infty_ball_constraint} instantiates the theoretical bounds for the $\infty$-ball constraint $Q=\Binf$. For quadratic regularization $\psi = \frac{1}{2}\norm{\cdot}^2$, we prove optimality of \Cref{loc:corollary_of_the_inward_shift_result.statement} and \Cref{loc:applied_margin_measure_for_exact_regularization.statement} up to a multiplicative constant, as confirmed by the lower bound in \Cref{loc:tightness_of_the_loo_ball_case.statement}. For linear regularization $\psi(x) = x \cdot p$, the margin-based approach yields tighter bounds for sparse vectors. We also apply the bounds to regularized optimal transport.
\end{enumerate}

Proofs of all mathematical statements in this paper appear in \Cref{loc:body.proofs}; auxiliary results appear in \Cref{sec:aux_proofs}.

\subsection{Recent work}
\label{loc:introduction_to_exact_regularisation_with_proof.recent_work}

Regularized optimal transport~\cite{dessein_RegularizedOptimalTransport_2018} has renewed interest in the regularization of LPs. Entropically-regularized optimal transport gained widespread adoption when \citet{cuturi_SinkhornDistancesLightspeed_2013} demonstrated that the Sinkhown algorithm would then apply and carried substantial computational advantages over the unregularized problem. \citet{cominetti_AsymptoticAnalysisExponential_1994} and \citet{weed_ExplicitAnalysisEntropic_2018} analyzed the convergence of entropyically regularized LPs as $\varepsilon \downarrow 0$. 

Entropic optimal transport produces \emph{dense} solutions that differ from the sparse solutions of the unregularized problem. Quadratic regularizers, by contrast, preserve sparsity, which makes quadratically regularized optimal transport~\cite{blondel_SmoothSparseOptimal_2018} an attractive alternative. \citet{nutz_QuadraticallyRegularizedOptimal_2025} and \citet{gonzalez-sanz_MonotonicityQuadraticallyRegularized_2025} provide an analysis of the sparsity properties of this approach. 

\citet{gonzalez-sanzQuantitativeConvergenceQuadratically2025} establish exact regularization in quadratically regularized optimal transport, a stronger property than sparsity preservation. Following \citet{weed_ExplicitAnalysisEntropic_2018}, they characterize the threshold $\overline{\varepsilon}$ for quadratically regularized LPs as 
\begin{equation*}
\frac{1}{\overline{\varepsilon}} = 2 \, \max
\Set{\frac{x^\psi\cdot(x^\psi-x)}{g\cdot (x^\psi-x)} | x \in \vrt(Q)\backslash\mathrm{Sol}(P^\psi)},
\end{equation*}
where $x^\psi$ denotes the unique solution to the selection problem $P^\psi$ with $\psi=\frac{1}{2}\|\cdot\|^2$. This characterization requires an exhastive search over all vertices, which is computationally prohibitive. 

\citet{gonzalez-sanzQuantitativeConvergenceQuadratically2025} also derived a computable bound
\begin{equation}
\label{eq:gap_based_bound}
\overline{\varepsilon} \geq \Delta/(2BD)
\end{equation}
where $B=\sup_{x \in Q}\|x\|$, $D=\sup_{x,x' \in Q} \|x-x'\|$, and 
\begin{equation}
\label{eq:suboptimality_gap}
  \Delta = \min\Set{-g\cdot x| x \in \vrt(Q)\backslash\mathrm{Sol}(P_{0})} \,-\,
  p^*
\end{equation}
denotes the suboptimality gap. This bound, however, proves loose in the average case. Consider the problems $P_{0}$ and $P_{\varepsilon}$ with standard normal cost vector $-g$ and quadratic regularizer $\psi=\frac{1}{2}\|\cdot\|^2$, solved over $Q=\Binf$. The suboptimality gap here simplifies to $\Delta= 2 \min_{j \in[n]} |g_{j}|$ almost surely, with $c/n \leq \mathbb{E}\Delta \leq C/n $ for absolute constants $c,C >0$.
Because $B=\sqrt{ n }$ and $D=2\sqrt{ n }$, bound~\eqref{eq:gap_based_bound} yields
\begin{equation*}
\mathbb{E}\,\overline{\varepsilon} \geq  c/(4n^2).
\end{equation*}
Our results (see \cref{loc:corollary_of_the_inward_shift_result.statement}), by contrast, yield 
\begin{equation*}
\mathbb{E}\,\overline{\varepsilon} \geq 1/(4n),
\end{equation*}
which improves the bound by a factor of $n$, and is optimal up to constants (see \cref{loc:tightness_of_the_loo_ball_case.statement}).

Exact regularization has been studied extensively in linear inverse problems, particularly for sparse and low-rank recovery. 
\citet{lai_AugmentedL1NuclearNorm_2013}, \citet{schopfer_ExactRegularizationPolyhedral_2012}, and \citet{zhang_LowerBoundGuaranteeing_2011} obtained explicit, computable estimates for the regularization threshold $\overline{\varepsilon}$ in these settings.
These results require a ground truth $x_{0}$ with specific structure (sparsity or low rank). The sparsity levels they assume, however, are far lower than typical of random basis pursuit instances. In contrast, we make no a priori assumptions on the solution set. Our average-case analysis of exact regularization therefore addresses a fundamentally different regime and uses different techniques.

\section{Inner cone bounds}
\label{loc:body.main_results}

The inner cone $N(z)\cap[N(z)+\partial(\varepsilon\psi)(z)]$, which appears in  \Cref{loc:geometric_probabilistic_results.statement}, simplifies under the \emph{membership condition} $N(z)\cap\partial(\varepsilon\psi)(z)\ne\emptyset$. Indeed, if $w\in N(z)\cap\partial(\varepsilon\psi)(z)$ then $N(z) \cap[N(z)+w] = N(z)+w$ by convexity of normal cones; see \Cref{fig:pubfig:250713:membership:representer.pdf}, left panel. \Cref{loc:body.main_results.a_simple_bound_under_a_membership_condition} leverages \Cref{loc:gaussian_measure_of_shifted_cones.statement} to obtain simple, concrete bounds when this membership condition holds. These bounds apply to interesting problems such as quadratically regularized optimal transport.

Generalizing beyond the restrictive membership condition requires further tools. \Cref{loc:body.main_results.bounds_via_representer_vectors} introduces the notion of \emph{representer vectors}. For a nondegenerate vertex $z$, we show that one can replace $w\notin N(z)$ with a different \emph{representer} vector $\tilde{w} \in N(z)$ that produces the same inner cone \eqref{eq:inner_cone_grad}; see \cref{fig:pubfig:250713:membership:representer.pdf}, right panel. Representer vectors allow us to port results from \cref{loc:body.main_results.a_simple_bound_under_a_membership_condition}.

\begin{figure}[t]
\centering
\includegraphics[width=0.4\textwidth]{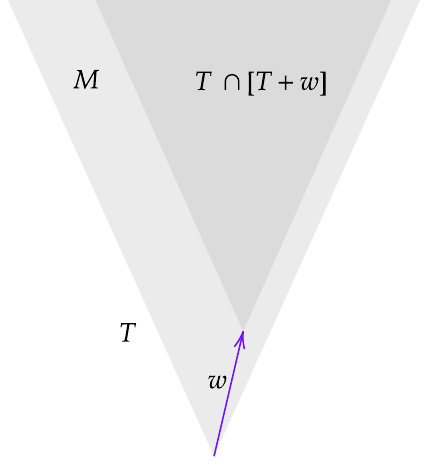}
\includegraphics[width=0.4\textwidth]{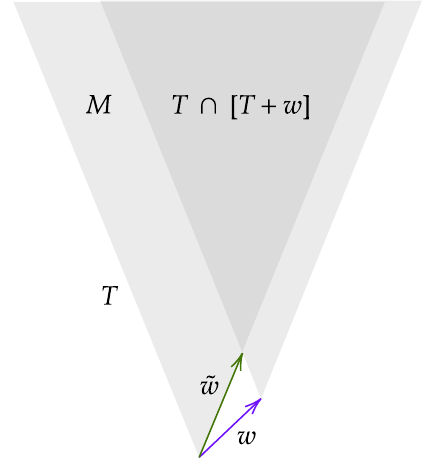}
\caption{ \textbf{The inner cone, margin, membership condition, and representer vectors}. Given a normal cone $T$ and a vector $w$, the corresponding inner cone $T \cap [T+w]$ coincides with the shifted cone $T+w$ when the membership condition $w\in T$ holds (left panel). When the membership condition fails (right panel), we can construct a ``representer" vector $\tilde{w}\in T$ so that the inner cone equals $T+\tilde{w}$. The margin $M(T,w)$ is the set of vectors that exit $T$ under the translation $w$.
\label{fig:pubfig:250713:membership:representer.pdf}}
\end{figure}

\subsection{A simple bound under a membership condition}
\label{loc:body.main_results.a_simple_bound_under_a_membership_condition}

We establish lower and upper bounds on the Gaussian measure of inner cones. The lower bound only requires convexity of $\psi$; the upper bound requires differentiability.

\begin{theorem}[Gaussian measure of an inner cone under the membership condition]
\label{loc:gaussian_measure_of_the_margin_resulting_from_an_inward_shift.statement}
Let $z \in \vrt(Q)$, $\varepsilon \geq 0$, and $T=N(z)$. For any subgradient $w \in \partial (\varepsilon\psi)(z) \cap T$,
\begin{equation*}
\gamma(T \cap [T + \partial (\varepsilon\psi)(z)]) \geq \gamma(T)\exp\left( - \tfrac{1}{2}\|w\|^2 - \|w\|\sqrt{n} \right).
\end{equation*}
If $\psi$ is differentiable at $z$, then
\begin{equation*}
\gamma(T \cap [T +\varepsilon \nabla \psi(z)]) \leq \gamma(T) \exp \left( - \tfrac{1}{2} \|\proj_{T^*} \,\varepsilon\nabla \psi(z)\|^2 + \mathrm{dist}(\varepsilon\nabla \psi(z), T^*)\sqrt{ n } \right).
\end{equation*}
\end{theorem}

The following corollary provides explicit bounds on both the failure probability of exact regularization and the expected regularization threshold $\mathbb{E}\,\overline{\varepsilon}$.

\begin{corollary}[Failure probability under the membership condition]
\label{loc:corollary_of_the_inward_shift_result.statement}
Suppose \cref{eq:regularity_relint} holds and $\varepsilon \geq 0$. Suppose that for each vertex $z \in \vrt(Q)$, we have $\partial \psi(z)\cap N(z) \neq \emptyset$. Choose any $v_{z}\in\partial \psi(z)\cap N(z)$ for each vertex, and define $B = \max\set{ \|v_{z}\| | z \in \vrt(Q) }$. Then
\begin{equation*}
\mathbb{P}[\mathrm{ER}^c(\varepsilon)] \leq 1- \exp \left( -\tfrac{1}{2}\varepsilon^2 B^2 -\varepsilon B\sqrt{ n } \right) \leq \tfrac{1}{2}\varepsilon^2 B^2 + \varepsilon B \sqrt{ n }.
\end{equation*}
The regularization threshold $\overline{\varepsilon}$ satisfies
\begin{equation*}
\mathbb{E}\,\overline{\varepsilon}
\geq \frac{1-e^{ -4n }}{2B\sqrt{ n }}.
\end{equation*}
\end{corollary}

\subsubsection{Quadratic regularization and inscribed polytopes}
\label{loc:corollary_of_the_inward_shift_result.example}
The following example satisfies the assumptions of \Cref{loc:corollary_of_the_inward_shift_result.statement}. Given distinct points $z_{1}, \dots, z_{K}$ on the (centered) sphere $\rho\,\mathbb{S}^{n-1}$ of radius $\rho>0$, consider the inscribed polytope $Q=\conv \{ z_{1},\dots,z_{K} \}$. Let $\psi= \frac{1}{2}\|\cdot\|^2$ be the quadratic regularizer. Then, at any vertex $z_{k}$, optimality requires $\nabla \psi(z_{k})=z_{k} \in N_{Q}(z_{k})$, so that the membership condition holds. By construction, $B=\max_{k\in[K]}\|\nabla \psi(z_{k})\|=\rho$, and the bound from \Cref{loc:corollary_of_the_inward_shift_result.statement} applies. Somewhat surprisingly, the bound does not depend on the number of vertices $K$.

The $1$-ball and the $\infty$-ball in $\mathbb{R}^n$ are covered by this example, with $\rho=1$ and $\rho=\sqrt{ n }$, respectively. The \emph{Birkhoff polytope} of doubly stochastic matrices
\begin{equation*}
Q = \Set{ X \in \mathbb{R}^{d\times d} | X\mathbf{e}=\mathbf{e},\, X^T\mathbf{e}=\mathbf{e},\,X \geq 0 },
\end{equation*}
where $\mathbf{e}=(1,\dots,1) \in \mathbb{R}^d$, also satisfies the conditions of \Cref{loc:corollary_of_the_inward_shift_result.statement}. This yields bounds for \emph{quadratically regularized optimal transport} when we identify $\mathbb{R}^{d \times d}$ with $\mathbb{R}^{d^2}$. The Birkhoff--von Neumann theorem \citep[Example 0.12]{ziegler_LecturesPolytopes_1995} characterizes $Q$ as the convex hull of all $d \times d$ permutation matrices, each of which lies on the Frobenius sphere $\{ X \in \mathbb{R}^{d\times d} \mid \|X\|_{F} = \sqrt{ d } \}$ of radius $\rho=\sqrt{ d }$. Because each vertex satisifes the sphere constraint, \Cref{loc:corollary_of_the_inward_shift_result.statement} applies with $n=d^2$ and $B=\sqrt{ d }=n^{1/4}$, which yields
\begin{equation}
\label{eq:birkhoff_prob_bound}
\mathbb{P}(\mathrm{ER}^c(\varepsilon)) \leq 1 - \exp\left( -\tfrac{1}{2}\varepsilon^2\sqrt{ n } - \varepsilon n^{3/4} \right) \leq \tfrac{1}{2}\varepsilon^2 \sqrt{ n } + \varepsilon n^{3/4}
\end{equation}
and
\begin{equation}
\label{eq:birkhoff_expec_bound}
\mathbb{E}\,\overline{\varepsilon} \geq \frac{1-e^{ -4n }}{2n^{3/4}}.
\end{equation}
We validate these bounds empirically by solving $P_{0}$ and $P_{\varepsilon}$ over the Birkhoff polytope with quadratic regularization for 20 independent realizations of $g \sim N(0,I_{n})$ across various pairs $(n, \varepsilon)$. The left panel of \Cref{fig:pubfig:250731:birkhoff_prob_expec.pdf} displays the empirical failure probability  $\mathbb{P}(\mathrm{ER}^c(\varepsilon))$ as a gray scale over $n$-$\varepsilon$ axes. The green dash-dot line $\varepsilon n^{3/4}=2$ represents an approximate level set of our bound~\eqref{eq:birkhoff_prob_bound}. Because $\frac{1}{2}\varepsilon^2\sqrt{ n }< 0.04 \cdot \varepsilon n^{3/4}$ holds over the tested values, the bound yields $\mathbb{P}(\mathrm{ER}^c(\varepsilon)) \leq 1-\exp(-2.08) \simeq 0.875$ on this line. The empirical probability on this line is approximately 0.5, consistent with the bound.

The right panel displays the expected threshold $\mathbb{E}\,\overline{\varepsilon}$, estimated by identifying the smallest tested $\varepsilon$ at which exact regularization fails for each realization of $g$, and averaging these values. We plot this empirical average of $\overline{\varepsilon}$ together with the 95\% confidence interval band. The green dash-dot line shows the lower bound~\eqref{eq:birkhoff_expec_bound}, without the negligible term $e^{ -4n }$. Both panels show that our bounds accurately capture the scaling relationship between $\varepsilon$ and $n$, observed as slopes in the log-scale plots.
\begin{figure}[t]
\centering
\includegraphics[width=\textwidth]{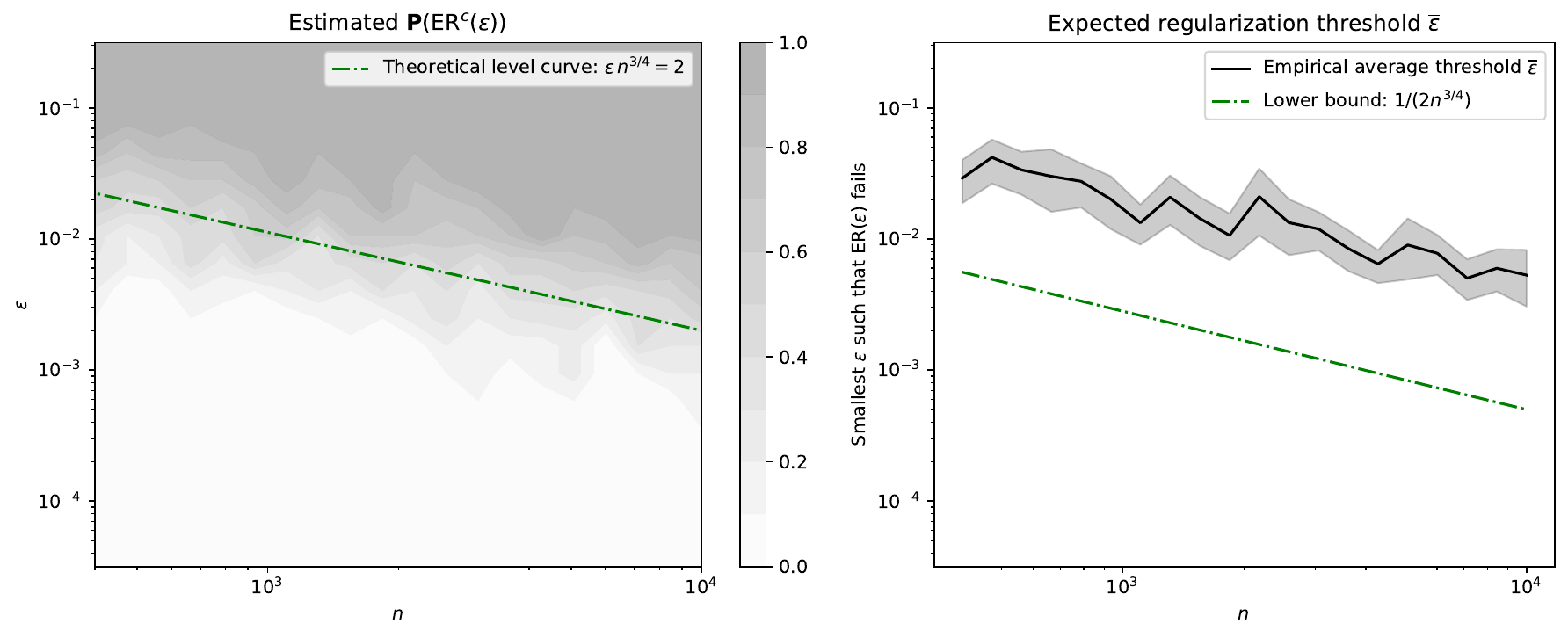}
\caption{\textbf{Probably exact regularization on the Birkhoff polytope}. The \textbf{left panel} displays the empirically failure probability of exact regularization as a colour scale over $n$-$\varepsilon$ axes, with the green dash-dot line showing the approximate level set $\varepsilon n^{3/4}=2$. The \textbf{right panel} displays the empirical average of the regularization threshold $\overline{\varepsilon}$ against $n$ with 95\% confidence intervals, together with our lower bound~\eqref{eq:birkhoff_expec_bound} as a green dash-dot line.
\label{fig:pubfig:250731:birkhoff_prob_expec.pdf}}
\end{figure}

\subsubsection{The soft phase transition}
\Cref{loc:gaussian_measure_of_the_margin_resulting_from_an_inward_shift.statement}
describes the order of $\varepsilon$ at which exact regularization
occurs. The phenomenon exhibits a soft phase transition:
exact regularization succeeds with high probability for sufficiently small
$\varepsilon$, and fails with high probability for
large $\varepsilon$. The next corollary makes this transition explicit.

\begin{corollary}[Soft phase transition]
\label{loc:basic_exact_regularisation_theorem.statement}
Let $\psi$ be differentiable at all $z  \in \vrt(Q)$ with $\nabla \psi(z) \in N(z)$.
If
\begin{equation*}
\varepsilon \leq \delta\left(2\sqrt{n}\max\set{\|\nabla \psi(z)\| | z \in \vrt(Q)}\right)^{-1},
\end{equation*}
then $\ER(\varepsilon)$ happens with probability at least $1-\delta$. If instead $\nabla \psi(z) \in N^\star(z)$ for all $z \in\vrt(Q)$  and
\begin{equation*}
\varepsilon  \ge \sqrt{2\log(1/\delta)}\cdot\left(\min\set{\|\nabla \psi(z)\| | z \in \vrt(Q)}\right)^{-1},
\end{equation*}
then $\mathrm{ER}^c(\varepsilon)$ happens with probability at least $1-\delta$.
\end{corollary}
One could derive analogous corollaries for other results in this paper, notably \Cref{loc:bound_via_representer_vectors.statement} and \Cref{loc:applied_margin_measure_for_exact_regularization.statement}, but we omit them to avoid redundancy.

\subsection{Bounds via representer vectors}
\label{loc:body.main_results.bounds_via_representer_vectors}
Let $T=N(z)$ be the normal cone at $z$. For a vector $w \in \partial(\varepsilon\psi)(z)$ that \emph{fails} the membership condition, we construct a \emph{representer vector} $\tilde{w}$ satisfying the membership condition such that
$$
T \cap [T+w] = T \cap [T+\tilde{w}].
$$
This construction relies on an explicit polyhedral description of $T$ via its facet structure, which we develop below. The main result, given in \Cref{loc:bound_via_representer_vectors.statement}, provides probabilistic bounds in terms of the norm $\norm{\tilde{w}}$.

\subsubsection{Representer vector construction}

For any subgradient $w \in \partial(\varepsilon \psi)(z)$, the inclusion
\begin{equation*}
T \cap [T+\partial(\varepsilon \psi)(z)] \supseteq T \cap [T+w]
\end{equation*}
reduces the problem to lower bounding $\gamma(T\cap[T+w])$. 

We define the margin $M(T,w) := T\setminus[T+w]$, where $w$ may be a vector or a set; in the latter case, the sum is a Minkowski sum. We use the margin to bound the complement of the inner cone; upper bounds on $\gamma(M(T, w))$ give lower bounds on the probability of exact regularization.

We show that the representer vector $\tilde{w}$ satisfies $M(T,w)=M(T,\tilde{w})$, as illustrated in \cref{fig:pubfig:250713:membership:representer.pdf}(b). This construction depends on the explicit polyhedral description of $T$, specifically via its inward-facing unit normals
$s^1, \ldots, s^{\ell} \in \mathbb{S}^{n-1}$:
\begin{equation}\label{eq:inward-normal-description}
T  =  \set{y \in \mathbb{R}^n | y  \cdot s^i  \ge 0 \,\,\forall i  \in [\ell]}.
\end{equation}
Each unit normal $s^i$ defines a facet $T^i = \set{ y \in T | s^i \cdot y = 0 }$, which is an $(n-1)$-dimensional polyhedral face of $T$. We denote by $S \in \mathbb{R}^{\ell \times n}$ the matrix with rows $s^1,\dots,s^\ell$ (or $S_T$ when emphasizing the dependence on $T$).  When the vertex $z$ is nondegenerate, the submatrix $A_{J(z)}$ with rows $\set{ a_{i} | i \in J(z) }$ is square and invertible, and we can express $T=A_{J(z)}^T \mathbb{R}^n_+ = \set{y \mid A^{-T}_{J(z)}y \in \mathbb{R}^n_+}$; cf.\@ \cref{loc:introduction_to_exact_regularisation_with_proof.mathematical_set:up_and_the_central_theorem}. Consequently, the rows $s^1,\dots,s^n$ (where $\ell=n$) can be taken to be the normalized columns of $A_{J(z)}^{-1}$, which implies that $S_{T}$ is invertible.

Given $S_{T}$ and $w$, we define a representer vector as any vector $\tilde{w}\in \mathbb{R}^n$ that satisfies
\begin{equation*}
S_{T} \tilde{w} = (S_{T}w)_{+},
\end{equation*}
where $x_{+}:= \max(0,x)$ operates componentwise. Any solution $\tilde{w}$ that satisfies this equation belongs to $T$ because $s^i\cdot \tilde{w} = (S_{T}\tilde{w})_{i} \geq 0$ for all $i \in [\ell]$; cf.~\eqref{eq:inward-normal-description}. Nondegeneracy of the vertex $z$ guarantees that such a representer exists..The following lemma confirms that $\tilde{w}$ generates the correct margin.
\begin{lemma}[Representer vectors for the margin]
\label{loc:representer_vectors_for_the_margin.statement}
If $\tilde{w}$ is a representer vector satisfying $S_{T}\tilde{w}=(S_{T}w)_{+}$, then it holds that
$M(T,w)=M(T, \tilde{w})$.
\end{lemma}
See \hyperlink{loc:representer_vectors_for_the_margin.proof_(representing_property)}{the proof} in \Cref{loc:body.proofs}.

\subsubsection{Representer vector theorem}
We now present the analogous result to \Cref{loc:gaussian_measure_of_the_margin_resulting_from_an_inward_shift.statement}.
\begin{theorem}[Bound via representer vectors]
\label{loc:bound_via_representer_vectors.statement}
Let $z \in \vrt(Q)$, $\varepsilon \geq 0$, and denote $T=N(z)$. Assume that the vertex $z$ is nondegenerate, so that the matrix $S_{T}$ is invertible. Then, for any subgradient $w \in \partial (\varepsilon\psi)(z)$, it holds that
\begin{equation*}
\gamma(T \cap [T + \partial (\varepsilon\psi)(z)]) \geq \gamma(T)\exp\left( - \tfrac{1}{2}\|\tilde{w}\|^2 - \|\tilde{w}\|\sqrt{n} \right),
\end{equation*}
where $\tilde{w}:= S_{T}^{-1}(S_{T}w)_{+}$ is the corresponding representer vector. Additionally, if $\psi$ is differentiable at $z$,
\begin{equation*}
\gamma(T \cap [T +\varepsilon \nabla \psi(z)]) \leq \gamma(T) \exp \left( - \tfrac{1}{2} \|\proj_{T^*} \tilde{w}\|^2 + \mathrm{dist}(\tilde{w}, T^*)\sqrt{ n } \right),
\end{equation*}
where now $\tilde{w}:=  S_{T}^{-1}[S_{T} (\varepsilon\nabla \psi(z))]_{+}$.
\end{theorem}

The bounds in \Cref{loc:bound_via_representer_vectors.statement} depend on the norm of the representer $\tilde{w}$. With the representation $\tilde{w}=S_{T}^{-1}(S_{T}w)_{+}$, we bound this norm in terms of $w$ and the condition number $\mathrm{cond}(S_{T}):=\sigma_{1}(S_{T})/\sigma_{n}(S_{T})=\|S_{T}\|_{\mathrm{op}}\|S_{T}^{-1}\|_{\mathrm{op}}$, where $\sigma_{i}(S_{T})$ denotes the $i$th largest singular value of $S_{T}$:
\begin{align*}
\|\tilde{w}\|=\|S_{T}^{-1}(S_{T}w)_{+}\|
&\le \|S_{T}^{-1}\|_{\mathrm{op}} \|(S_{T}w)_{+}\|  \\
&\le \|S_{T}^{-1}\|_{\mathrm{op}}\,\|S_{T}\|_{\mathrm{op}}\,\|w\|
= \mathrm{cond}(S_{T})\,\|w\|.
\end{align*}

\section{Margin bounds via facet decomposition}
\label{loc:body.main_results.using_the_boundary_of_normal_cones}

We establish two-sided bounds on the Gaussian measure of the margin in terms of a geometric quantity.
\Cref{loc:geometric_probabilistic_results.statement} establishes
that
\begin{equation*}
 \prob(\mathrm{ER}^c(\varepsilon))  = \textstyle \sum_{z  \in \vrt(Q)}
\gamma(N(z) \setminus [N(z) + \partial(\varepsilon \psi)(z)]).
\end{equation*}
Because the margin~\eqref{eq:margin} appears in each
summand, we bound $\gamma(M(N(z), w))$ for each vertex $z$ and
subgradient $w \in \partial (\varepsilon \psi)(z)$. These margin
bounds exploit the facet structure of the normal cone boundary.

\subsection{Margin partition and Gaussian measure}

Because facets of normal cones lie in proper subspaces, we require a measure adapted to lower-dimensional sets.  For any cone $T  \subseteq \mathbb{R}^n$, the \emph{relative Gaussian measure} is $\gamma_{\mathrm{rel}}(T) := \prob(\proj_{\Span(T)} g  \in T)$, 
where $g \sim \mathcal{N}(0, I_n)$. This measures $T$ relative to its span, rather than the ambient space.
We use this facet decomposition to bound the margin $M(T, w)$.
\Cref{fig:pubfig:250713:margin:partition.pdf} illustrates the geometry underlying this decomposition.

\begin{lemma}[Partition of the margin]
\label{loc:partition_of_the_margin.statement}
For a polyhedral cone $T \subseteq \mathbb{R}^n$ with $\ell$ facets $T^i$ associated with inward normals $s^i$, and $w \in \mathbb{R}^n$,
\begin{equation}
\label{eq:partition}
\gamma(M(T,w)) \le \sum_{i=1}^\ell \gamma(P^i),
\quad\text{where}\quad
P^i :=
\begin{cases}
\bigcup_{t  \in [0, 1)} [T^i + tw] & \text{if $s^i  \cdot w > 0$},\\
\emptyset & \text{otherwise.}
\end{cases}
\end{equation}
If $w \in T$, then \eqref{eq:partition} holds with equality.
\end{lemma}

\begin{figure}[t]
\centering
\includegraphics[width=0.7\textwidth]{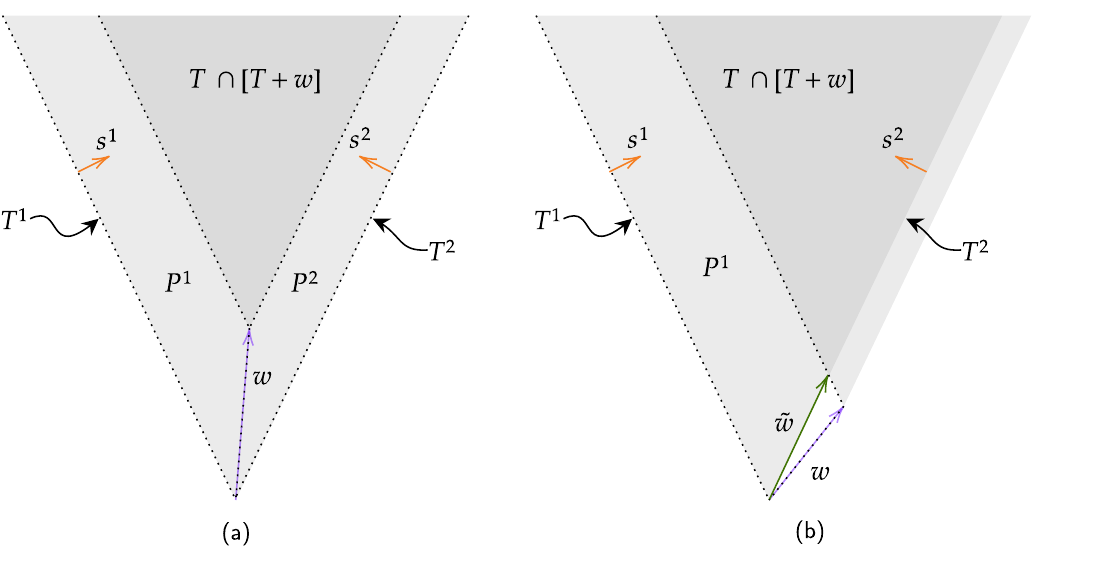}
\caption{ \textbf{Decomposing the margin}. Consider a polyhedral cone $T$ with facets $T^1$, $T^2$ and associated inward normals $s^1, s^2$. Given a vector $w$, the margin $M(T,w)$ lies within the union of $P^1$ and $P^2$ (regions enclosed by dotted lines). The containment is tight when $w \in T$ (panel a), but not in general (panel b).
\label{fig:pubfig:250713:margin:partition.pdf}}
\end{figure}

The partition in \Cref{loc:partition_of_the_margin.statement} allows for explicit Gaussian measure bounds.
\begin{lemma}[Gaussian measure of the margin]
\label{loc:gaussian_measure_of_the_margin.statement}
Let $T \subseteq \mathbb{R}^n$ be a polyhedral cone with $\ell$ facets $T^i$ with associated inward unit normals $s^i$. For $w \in \mathbb{R}^n$, let
\begin{equation} \label{eq:measure-of-margin}
F(T, w) := (\sqrt{2\pi})^{-1}\textstyle\sum_{i=1}^\ell  \gamma_{\mathrm{rel}}(T^i)(s^i \cdot w)_+.
\end{equation}
Then
\begin{equation*}
\gamma(M(T, w)) \le 2 F(T,w)\exp \left( \sqrt{n-1} \|w\|\right).
\end{equation*}
If $w  \in T$, then
\begin{equation*}
\gamma(M(T, w)) \ge F(T, w) \exp \left(- \sqrt{n-1}\|w \| - \tfrac12\|w\|^2\right).
\end{equation*}
\end{lemma}

\subsection{Application to exact regularization}

We now give a result for the probability of
exact regularization using the margin set in a normal cone of $Q$.

\begin{theorem}[Applied margin measure for exact regularization]
\label{loc:applied_margin_measure_for_exact_regularization.statement}
Let $z \in \vrt(Q)$. Let $T^i$ denote 
the facets of the polyhedral normal cone $N(z)$, with their associated inward
normals $s^i$. Let the function $F$ be as in \Cref{loc:gaussian_measure_of_the_margin.statement}. Then for any $v \in \partial \psi(z)$,
\begin{equation*}
\gamma(M(N(z),\, \partial (\varepsilon \psi)(z)))  \le 2\varepsilon F(N(z),  v)\exp \left( \varepsilon\sqrt{n-1} \|v\|\right).
\end{equation*}
Furthermore, if $\psi$ is differentiable at $z$ and $\nabla \psi(z) \in N(z)$, then
\begin{equation*}
\gamma(M(N(z),\, \varepsilon \nabla \psi(z)))  \ge \varepsilon F(N(z),   \nabla \psi(z)) \exp \left(- \varepsilon\sqrt{n-1}\| \nabla \psi(z) \| - \tfrac12 \varepsilon^2 \|\nabla \psi(z)\|^2\right).
\end{equation*}
Alternatively, if we allow $\nabla \psi(z) \notin N(z)$, but instead require that $N(z)$ be nondegenerate, then with
$\tilde{v} := S^{-1}(S \nabla \psi(z))_+$,
\begin{equation*}
\gamma(M(N(z),\, \varepsilon \nabla\psi(z)))  \ge \varepsilon F(N(z), \nabla \psi(z)) \exp \left(- \varepsilon \sqrt{n-1}\|\tilde{v}\| - \tfrac12\varepsilon^2\|\tilde{v}\|^2\right).
\end{equation*}
\end{theorem}

\subsubsection{Facets decomposition}

We use \Cref{loc:applied_margin_measure_for_exact_regularization.statement} to characterize the probability of the $\ER^c(\varepsilon)$ event for small $\varepsilon$ via the sum $\sum_{z \in \vrt(Q)} F(N(z),  \nabla \psi(z))$.
Indeed, with \Cref{loc:geometric_probabilistic_results.statement}, we can say that with $\psi$
differentiable and $\nabla\psi(z) \in N(z)$ for all $z \in \vrt(Q)$, and with
$\varepsilon \le (2\sqrt{n} \max_{z \in Q} \|\nabla \psi(z)\|)^{-1}$,
\begin{equation*}
 \textstyle \frac{\varepsilon}{e} \sum_{z \in \vrt(Q)} F(N(z),  \nabla \psi(z)) \le \prob(\mathrm{ER}^c(\varepsilon)) \le 2 e\varepsilon \sum_{z \in \vrt(Q)} F(N(z),  \nabla \psi(z)).
\end{equation*}

Using \cref{loc:applied_margin_measure_for_exact_regularization.statement}, 
we find a simple result that applies to the linear regularizer $\psi(x)= p \cdot x$, for
which we require some additional notation.
We henceforth refer to $1-$dimensional faces of $Q$ as ``edges", and denote by $\mathcal{B}$ the set of all edges of $Q$. Any edge $f \in \mathcal{B}$
takes the form $f=\conv(z_1, z_2)$ for $z_1, z_2  \in \vrt(Q)$.
The facets of the normal cones at vertices of $Q$ are themselves $(n-1)$-dimensional
normal cones at points in the relative interiors of edges, and
are also the intersection of exactly two full-dimenional normal cones
at vertices $z_1, z_2$ (see~\citep[Proposition 1(b)]{lu_NormalFansPolyhedral_2008}).
Because the normal cone is the same for any point in $\rint(f)$ for an edge $f \in \mathcal{B}$, we denote by $N(f)$
the normal cone associated with the edge $f$. Furthermore, an inward normal $s^i$ associated
with a facet is orthogonal to the facet, and so is $z_1 - z_2$.
Therefore, these two vectors are linearly dependent. Because of this connection,
we define $s_f$ to be a unit normal in $\Span(f-f)$. (There are two possible such vectors, 
we pick any one of them, as the orientation will not matter in the next result.)

\begin{corollary}[Linear perturbed polytope regularization]
\label{loc:linear_perturbed_polytope_regularization.statement}
With $\psi(x) = p \cdot x$, any polytope $Q \subseteq \mathbb{R}^n$, and $\mathcal{B}$ the set of edges 
of $Q$,
\begin{equation*}
\prob(\mathrm{ER}(\varepsilon)^c) \le \textstyle\varepsilon\frac{ e}{\sqrt{2\pi}} \sum_{f \in \mathcal{B}} \gamma_{\mathrm{rel}}(N(f))\,\,  |s_f  \cdot p|\, \,  \exp \left( \varepsilon\sqrt{n-1} \|p\|\right).
\end{equation*}
\end{corollary}
\begin{proof}[\hypertarget{loc:linear_perturbed_polytope_regularization.proof}Proof of \Cref{loc:linear_perturbed_polytope_regularization.statement}]
Denote by $T_z^i$ the $i^{\text{th}}$ facet of the normal cone $N(z)$,
and by $s_z^i$ the corresponding inward normal. 
As argued above, $s_z^i \approx s_f$ up to a sign, for $f \in \mathcal{B}$ the
edge for which $T_z^i = N(f)$. To apply \Cref{loc:applied_margin_measure_for_exact_regularization.statement},
we compute
\begin{align*}
\textstyle\sum_{z \in \vrt(Q)} F(N(z), \varepsilon p)
&= \textstyle\sum_{z \in \vrt(Q)}  \frac{1}{\sqrt{2\pi}}\sum_{i = 1}^{\ell_z} \gamma_{\mathrm{rel}}(T_z^i) (\varepsilon p \cdot s^i_z)_+\\
&= \textstyle\frac{\varepsilon}{\sqrt{2\pi}}\sum_{f \in \mathcal{B}} \gamma_{\mathrm{rel}}(N(f)) | p \cdot s_f|.
\end{align*}
As discussed above, each facet $T_z^i$ appears twice in the sum because 
each facet is a face for exactly two full-dimensional normal cones. 
When we pair the two terms related to a common edge $f$, the inward
normals in these two terms satisfy $s_{z_1}^i =  -s_{z_2}^j$ because they point inward to different
full-dimensional normal cones.
Therefore, the factor $(\varepsilon p \cdot s_z^i)_+$
will take a value of $\varepsilon |p \cdot s_z^i|$ for one of the two terms and $0$ for the other, which justifies the result.
\end{proof}
\section{\texorpdfstring{$\Binf$}{B-infinity} ball constraint}
\label{loc:body.b_infty_ball_constraint}
We analyze the case where $Q = \Binf$ to gain a better understanding of the bounds introduced so far. The polytope $\Binf$ resembles
the feasible region of linear programs because it can be specified by $2n$ linear constraints, and has an
exponential number of vertices. At the same time, it is 
amenable to analysis because of its simplicity.
We first consider the quadratic (Tikhonov) regularization, where $\psi(x) = \frac{1}{2}\|x\|^2$. We will assess the quality of our bounds in this setting by comparing them to the following lower-bound on
$\prob(\mathrm{ER}^c(\varepsilon))$.
\begin{proposition}[Lower bound for $\Binf$]
\label{loc:tightness_of_the_loo_ball_case.statement}
For $P_0$ with $Q = \Binf$ and
$\psi = \frac{1}{2}\|\cdot\|^2$,
\begin{equation*}
\prob(\mathrm{ER}^c(\varepsilon))  \ge 1-\exp\left(-\sqrt{\tfrac{2}{\pi e}} n \varepsilon \right),
\end{equation*}
and for $\varepsilon \leq \sqrt{\frac{(\pi e)}{2}} \frac{1}{2n}$,
\begin{equation*}
\prob(\mathrm{ER}^c(\varepsilon))  \ge \tfrac{1}{\sqrt{2\pi e}} n \varepsilon.
\end{equation*}
\end{proposition}

Since $\Binf$ has its vertices on $\sqrt{n} \sphere{n}$, and the gradient at each vertex $z$ is $z$ itself, the gradient is contained in $N(z)= \{ x  \in \mathbb{R}^n \mid \sgn(x) = \sgn(z)\}$.
So we may use \Cref{loc:corollary_of_the_inward_shift_result.statement} to
derive the following.
\begin{corollary}[Bound from vertices on the sphere]
\label{loc:boo_from_points_on_sphere.statement}
For problem $P_0$ with $Q = \Binf$ and
$\psi(x) = \frac{1}{2}\|x\|^2$,
\begin{equation*}
\mathbb{P}[\mathrm{ER}^c(\varepsilon)] \leq \varepsilon n + \tfrac{1}{2} \varepsilon^2 n.
\end{equation*}
\end{corollary}
It turns out that for $\varepsilon < n^{-1}$, the first term dominates, and this bound is optimal up to a constant factor of no more than $\sqrt{2\pi e}$.

What happens if we instead use \Cref{loc:applied_margin_measure_for_exact_regularization.statement}?
To apply this result, we need the relative Gaussian measure of the
facets of the normal cones of $\Binf$. The normal cones are the orthants of
$\mathbb{R}^n$, and their facets are orthants of
the axis-aligned hyperplanes. An axis aligned hyperplane is of the form
$\{x \in \mathbb{R}^n \mid x  \cdot  e_i = 0\}$ for a canonical basis vector $e_i$. 
There are exactly $n$ such hyperplanes,  and each hyperplane has exactly
$2^{n-1}$ orthants. Therefore, 
the relative Gaussian measure of each facet of the normal cones is
$2^{-(n-1)}$, by rotational symmetry. With $F$ as in \Cref{loc:gaussian_measure_of_the_margin.statement},
\begin{align*}
F\left(N(z), \varepsilon z\right) &=  \tfrac{1}{\sqrt{2\pi}}\textstyle\sum_{i = 1}^n \left(\frac{1}{2}\right)^{n-1} \left(\varepsilon z \cdot (z_i e_i)\right)_+
 =  \frac{1}{\sqrt{2\pi}}\sum_{i = 1}^n \left(\frac{1}{2}\right)^{n-1} \varepsilon z_i^2.
\end{align*}
Then,
\begin{equation}\label{eq:bound-from-margin}
\begin{aligned}
\prob(\mathrm{ER}^c(\varepsilon)) &\le 2\textstyle\sum_{z \in \{-1, 1\}^n} F\left(N(z), \varepsilon z\right)\exp \left( \varepsilon\sqrt{n-1}  \cdot  \sqrt{n}\right)\\
 & = 2  \cdot 2^n \left[\tfrac{1}{\sqrt{2\pi}} \left(\tfrac{1}{2}\right)^{n-1} \varepsilon \|z\|^2\right] \exp \left( \varepsilon\sqrt{n-1}  \cdot \sqrt{n}\right)\\
&\le \tfrac{4}{\sqrt{2\pi}} \varepsilon n \exp \left( \varepsilon n \right).
\end{aligned}
\end{equation}
We have derived the following.
\begin{corollary}[Bound from the margin]
\label{loc:boo_bound_from_margin.statement}
For problem $P_0$ with $Q = \Binf$ and
$\psi(x) = \frac{1}{2}\|x\|^2$,~\eqref{eq:bound-from-margin} holds.
\end{corollary}
This bound is optimal up to a constant. It has a better constant than the bound in \Cref{loc:boo_from_points_on_sphere.statement}, but with an additional exponential factor of approximately $1$ in the regime of interest: $\varepsilon < 1/n$.

We now turn our attention to the setting where $\psi(x) = p \cdot x$. In this
setting, the conditions of \Cref{loc:gaussian_measure_of_the_margin_resulting_from_an_inward_shift.statement} do not hold, and therefore we need the assistance of the representer vector in \Cref{loc:bound_via_representer_vectors.statement}. Because $Q=\Binf$, the representer vector $\tilde{p}$ of $p$ for a specific normal cone is in fact the projection of $p$ on to that normal cone. Because of this, $\|\tilde{p}\| \le \|p\|$, and with \Cref{loc:gaussian_measure_of_shifted_cones.statement},
\begin{align*}
\prob(\mathrm{ER}(\varepsilon)) &\ge \textstyle\sum_{z \in \{-1, 1\}^n} \gamma_{\mathrm{rel}}(N_{\Binf}(z)) \exp\left( - \frac12\varepsilon^2\|\tilde{p}\|^2 - \varepsilon\|\tilde{p}\|\sqrt{n} \right)\\
 & \ge   \exp\left( - \tfrac12\varepsilon^2\|p\|^2 - \varepsilon\|p\|\sqrt{n} \right) \ge  1  -  \tfrac12\varepsilon^2\|p\|^2 - \varepsilon\|p\|\sqrt{n}.
\end{align*}
\begin{corollary}[Bound with the representer vector]
\label{loc:linear_perturbed_boo_from_inner_cone.statement}
For problem $P_0$ with $Q = B_\infty$ and
$\psi(x) = p  \cdot x$,
\begin{equation*}
\prob(\mathrm{ER}^c(\varepsilon))  \le \varepsilon\|p\|\sqrt{n} + \tfrac12\varepsilon^2\|p\|^2.
\end{equation*}
\end{corollary}
We apply the bound in \Cref{loc:linear_perturbed_polytope_regularization.statement}
to the same setting. 
Recall that all $(n-1)$-dimensional normal cones of the $\Binf$ ball have a
relative Gaussian measure of $2^{-(n-1)}$, and that there are $n2^{n-1}$ such
normal cones (one for each edge). The edge vectors $s_f$ are the
canonical basis vectors. We find the following result.
\begin{corollary}[Linear perturbed $\Binf$ from the margin bound]
\label{loc:linear_perturbed_boo_from_margin.statement}
For problem $P_0$ with $Q = \Binf$ and
$\psi(x) = p  \cdot x$,
\begin{equation*}
\prob(\mathrm{ER}^c(\varepsilon)) \le \tfrac{1}{\sqrt{2\pi}}\varepsilon\|p\|_1\exp \left( \varepsilon\sqrt{n-1} \|p\|\right).
\end{equation*}
\end{corollary}

The use of \Cref{loc:linear_perturbed_polytope_regularization.statement} yields a tighter bound relative to the simpler representer-type machinery in this setting. We find a smaller multiplicative constant, but more importantly, the factor of $\lVert p\rVert_1$ is smaller than the factor of $\sqrt{n} \|p\|$ found in \Cref{loc:linear_perturbed_boo_from_inner_cone.statement}. \Cref{loc:linear_perturbed_boo_from_margin.statement} tells us that exact regularization occurs more readily for almost-sparse vectors $p$, i.e., vectors that have their mass concentrated on a small number of entries, as quantified by the ratio $\|p\|_1/(\sqrt{n}\|p\|)$. 

To conclude this section, we repeat the numerical experiments from \Cref{loc:corollary_of_the_inward_shift_result.example} (\Cref{fig:pubfig:250731:birkhoff_prob_expec.pdf}), this time for $Q=\Binf$, with both quadratic $\psi(x)=\frac{1}{2}\|x\|^2$ and linear $\psi(x)=p\cdot x$ regularizations. The vector $p$ is chosen uniformly at random over the sphere $\sqrt{ n }\mathbb{S}^{n-1}$.  For various pairs of $n$ and $\varepsilon$, the problems $P_{0}$ and $P_{\varepsilon}$ are solved for 20 independent realizations of the objective vector $g \sim N(0,I_{n})$, and the quantity $\mathbb{P}(\mathrm{ER}^c(\varepsilon))$ is estimated empirically as the proportion of realizations for which exact regularization fails. Estimated $\mathbb{P}(\mathrm{ER}^c(\varepsilon))$ is plotted over $n$-$\varepsilon$ axes for quadratic (left panel) and linear (right panel) regularizations. The green dash-dot lines in both panels describe approximate level curves for our theoretical bounds, noting that $\varepsilon n$ is the dominant term in the bounds from \Cref{loc:boo_from_points_on_sphere.statement}, \Cref{loc:boo_bound_from_margin.statement} and \Cref{loc:linear_perturbed_boo_from_inner_cone.statement} when $\varepsilon$ is small and $\|p\|=\sqrt{ n }$. As before, both panels show that our theoretical results capture the $\varepsilon$-$n$ relationship (i.e., slopes in the log-scale plots) quite accurately.
\begin{figure}[t]
\centering
\includegraphics[width=\textwidth]{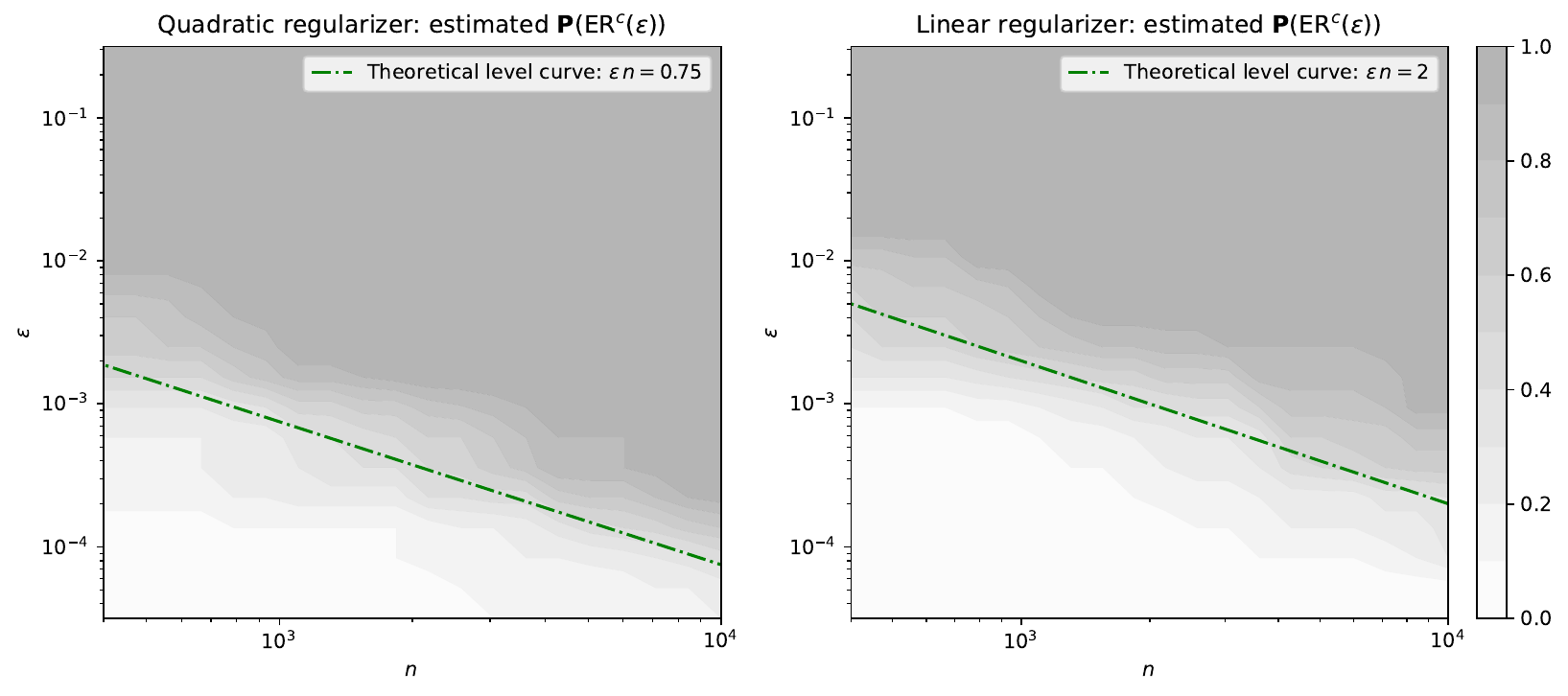}
\caption{\textbf{Probably exact regularization on the $\Binf$ polytope}. The left and right panels plot the empirically observed probability of \emph{failure} of exact regularization over 20 independent trials as a colour scale over $n$-$\varepsilon$ axes, for quadratic and linear regularizations respectively. Approximate level curves for our theoretical bounds are included as green dash-dot lines.
\label{fig:pubfig:250731:contourf:binfty:both:reg:single:colorbar:flipped.pdf}}
\end{figure}
\section{Proofs}
\label{loc:body.proofs}
\begin{proof}[\hypertarget{loc:normal_cones_and_faces.proof}Proof of \Cref{loc:normal_cones_and_faces.statement}]
(1) Recalling the index-set of active constraints $J(z):= \{i \in[m] \mid a_{i}\cdot z= b_{i} \}$ at $z \in Q$, we may express the normal cone as
$N(z) = \conv\cone \set{ a_{i} | i \in J(z) }$.
Now, $\dim N(z) = \dim\,\mathrm{aff}\,N(z)
= \dim \mathrm{span}\,\{ a_{i} \mid i \in J(z) \}.$

This dimension equals $n$ if and only if $z \in \vrt(Q)$.

(2) By \Cref{loc:auxiliary_result_intersections_of_normal_cones.statement}, $N(x) \cap N(z) = N_{Q}\left( \frac{1}{2}(x+z) \right)$. But $\frac{1}{2}(x+z)$ cannot be a vertex of $Q$, and hence by (1), $N_{Q}\left( \frac{1}{2}(x+z) \right)$ is not $n$-dimensional. To prove $N(x) \cap \mathrm{int}\,N(z) = \emptyset$, we use the fact that the intersection $[x + N(x)] \,\cap \, [z + N(z)]$ is empty when $x \neq z$ (this is a consequence of the Euclidean projection being single-valued \cite[Section A.5.3]{hiriart-urruty_FundamentalsConvexAnalysis_2001}). Now suppose, towards contradiction, that there exists a vector $v$ in $N(x) \cap \mathrm{int}\,N(z)$. Then one can find a radius $\rho >0$ such that $v+ \rho \mathbb{B} \subseteq N(z)$, where $\mathbb{B}$ is the Euclidean unit ball in $\mathbb{R}^n$. But because $N(z)$ is a cone, we may scale as $tv + t\rho \mathbb{B} \subseteq N(z)$ for any $t >0$, and in particular, pick a $t$ such that $\|x-z\| \leq t\rho$. In that case, we end up with the inclusion $x+tv \in [x + N(x)] \,\cap \, [z + N(z)]$ because
\begin{equation*}
x+tv = z + (x-z) + tv 
\in z + tv + t\rho \mathbb{B}
\subseteq z + N(z).
\end{equation*}
This contradicts the emptiness of the intersection.

(3) If $Q$ is a polytope, we may express it as $Q=\mathrm{conv}(\vrt(Q))$ \cite[Proposition 2.2]{ziegler_LecturesPolytopes_1995}. Given any $v \in \mathbb{R}^n$, consider the program
\begin{equation*}
\minimize \set{-v\cdot x | x \in Q}. \tag{\ensuremath{P_{v}}}
\end{equation*}
Because $\min_{x \in Q} (-v\cdot x)=\min_{z \in \vrt(Q)}(-v\cdot z)$, there exists a vertex $\overline{z}_{v} \in \vrt(Q)$ that solves ($P_{v}$). As a result, $\overline{z}_{v}$ satisfies the optimality condition $v \in N_{Q}(\overline{z}_{v})$, and it follows that $v \in \cup_{z \in \vrt(Q)} N(z)$ as desired.
\end{proof}%

\begin{proof}[\hypertarget{loc:uniqueness_of_solution_in_a_random_linear_program.proof}Proof of \Cref{loc:uniqueness_of_solution_in_a_random_linear_program.statement}]

For any vertex $z \in \vrt(Q)$, the normal cone $N(z)$ is polyhedral so that its boundary is contained in the union of finitely many hyperplanes. Noting that hyperplanes in $\mathbb{R}^n$ are $n-1$ dimensional, and consequently have zero Gaussian measure, we infer that $\gamma(\mathrm{bd}\,N(z))=0$ and $\gamma(\mathrm{int}\,N(z)) = \gamma(N(z))$. Because \Cref{loc:normal_cones_and_faces.statement} establishes that the interior normal cones partition the space,
\begin{align*}
1 \geq\gamma \left( {\dot\cup_{z \in \vrt(Q)}} \mathrm{int}\,N(z)\right)
&= \textstyle\sum_{z \in \vrt(Q)} \gamma(\mathrm{int}\,N(z)) \\
&\geq \gamma \left( \cup_{z\in\vrt(Q)}N(z) \right) = \gamma(\mathbb{R}^n) = 1,
\end{align*}
so that
$%
\mathbb{P}\left( {\dot\cup_{{z \in \vrt(Q)}}}\{ g \in \mathrm{int}\,N(z) \} \right) = 1.
$ %
Therefore, with probability 1 (on the realizations of $g$), there exists a vertex $z^\star \in \vrt(Q)$ such that $g \in \mathrm{int}\,N(z^\star)$. Such an $z^\star$ is clearly a solution of $P_{0}$. Moreover, it is also the unique solution because $\mathrm{int}\,N(z^\star) \cap N(x) = \emptyset$ for any other $x \in Q$ so that $g \notin N(x)$.

One can express the solution as a sum of Bernoulli indicators
\begin{equation*}
z^\star = \textstyle\sum_{z \in \vrt(Q)} z\,\indicator_{\{ g \in \mathrm{int}\,N(z) \}}
\end{equation*}
which is clearly a (measurable) random vector. Also,
\begin{equation*}
\mathbb{P}\{ z^\star = z \} = \mathbb{P}\left\{ g \in \mathrm{int}\,N(z) \right\} = \gamma(\mathrm{int}\,N(z)) = \gamma(N(z)).
\end{equation*}
\end{proof}

\begin{proof}[\hypertarget{loc:gaussian_measure_of_shifted_cones.proof}Proof of \Cref{loc:gaussian_measure_of_shifted_cones.statement}]

\textbf{Upper bound}.
We compute the Gaussian measure of $V + w$ for any cone $V \subseteq \mathbb{R}^d$. Let $\phi(x)=(2\pi)^{-d/2}\exp\left( -\|x\|^2 /2 \right)$ denote the standard Gaussian density, and $\phi^\kappa(x)=(2\pi \kappa)^{-d/2}\exp(-\|x\|^2 /2\kappa)$ be its scaled version with variance $\kappa>0$. Then,
\begin{align*}
\gamma(V+w) &= \int_{V+w} \phi(x)dx = \int_{V}\phi(x+w)dx
=\int_{V} \frac{\phi(x+w)}{\phi^\kappa(x)}\,\phi^\kappa(x)dx,
\end{align*}
where we have changed the base measure to the scaled Gaussian $d\gamma^\kappa(x) = \phi^\kappa(x)dx$. 

\textbf{Step 1: Hölder's inequality and scaling invariance}.
Fixing some $\kappa>0$ (to be chosen later), we upper bound the integral above as
\begin{equation*}
\int_{V} \frac{\phi(x+w)}{\phi^\kappa(x)}\,\phi^\kappa(x)dx \leq \left[ \sup_{x \in V} \frac{\phi(x+w)}{\phi^\kappa(x)} \right] \int_{V}\phi^\kappa(x)dx.
\end{equation*}
Due to the scaling invariance of cones (i.e., $\alpha V=V$ for $\alpha>0$),
\begin{align*}
\int_{V}\phi^\kappa(x)dx = \mathbb{P}_{g \sim \mathcal{N}(0,I)}\{ \sqrt{ \kappa }\,g \in V \}
= \mathbb{P}_{g\sim \mathcal{N}(0,I)}\{ g \in V \} = \gamma(V).
\end{align*}
What remains is to upper-bound the ratio $\phi(x+w)/\phi^\kappa(x)$, which we write out as
\begin{equation*}
\kappa^{d/2} \frac{\exp\left(-\tfrac12\|x + w\|^2\right)}{\exp\left(-\tfrac{1}{2\kappa}\|x\|^2\right)} = \kappa^{d/2} \exp\left(-Q_{w}(x)\right),
\end{equation*}
where $Q_{w}(x):= \frac{1}{2}\|x+w\|^2 - \frac{1}{2\kappa}\|x\|^2$, and consequently,
\begin{equation*}
\sup_{x \in V}\, \frac{\phi(x+w)}{\phi^\kappa(x)} = \kappa^{d/2}\exp \left( -\inf_{x \in V}\,Q_{w}(x) \right).
\end{equation*}
We proceed by lower-bounding $\inf_{x \in V}\,Q_{w}(x)$.

\textbf{Step 2: Lower bounds from weak duality}.
If the variance $\kappa > 1$, the quadratic $Q_{w}$ is strongly convex and has a finite infimum over the cone $V$. 
By weak duality \citep[Lemma 11.38]{rockafellar1998variational}, $\inf_{x \in V}\,Q_{w}(x) \geq \sup_{y \in V^*}\,-Q_{w}^*(y)$,
where the convex conjugate of the quadratic $Q_{w}$ is $Q_{w}^*(y) = \tfrac{\kappa}{2(\kappa-1)} \|y - w\|^2 - \tfrac{1}{2}\|w\|^2$.
We can now evaluate
\begin{align*}
\sup_{y \in V^*}\,-Q_{w}^*(y) &= \sup_{y \in V^*}\,\left\{ -\frac{\kappa}{2(\kappa-1)}\|y-w\|^2 + \tfrac{1}{2}\|w\|^2 \right\} \\
&= \tfrac{1}{2}\|w\|^2 - \frac{\kappa}{2(\kappa-1)}\inf_{y \in V^*} \|y-w\|^2 = \tfrac{1}{2}\|w\|^2 - \frac{\kappa}{2(\kappa-1)}\,\mathrm{dist}(w, V^*)^2,
\end{align*}
and consequently, express
\begin{align*}
\sup_{x \in V}\, \frac{\phi(x+w)}{\phi^\kappa(x)} &= \kappa^{d/2}\exp \left( -\inf_{x \in V}\,Q_{w}(x) \right) \\
&\leq \kappa^{d/2} \exp \left( \frac{\kappa}{2(\kappa-1)}\,\mathrm{dist}(w, V^*)^2 \right)\,\exp\left( - \tfrac{1}{2} \|w\|^2 \right).
\end{align*}

\textbf{Step 3: Optimize over $\kappa$}.
Define the function
\begin{equation*}
h(\kappa) = \kappa^{d/2} \exp\left( \frac{\kappa}{2 (\kappa - 1)} \mathrm{dist}(w, V^*)^2 \right),
\end{equation*}
and recall from the previous steps that
$%
\gamma(V+w) \leq \gamma(V)\,h(\kappa) \exp\left( - \tfrac{1}{2}\|w\|^2 \right)
$ %
holds for any choice of $\kappa > 1$.
We choose the value
\begin{equation*}
\kappa = 1 + \mathrm{dist}(w, V^*)/\sqrt{ d },
\end{equation*}
which is valid provided $\mathrm{dist}(w, V^*)>0$. (The case $\mathrm{dist}(w, V^*)=0$ will be handled separately.) This choice yields
\begin{align*}
h(\kappa) &= \left( 1+ \frac{\mathrm{dist}(w, V^*)}{\sqrt{ d }} \right)^{d/2} \exp \left\{ \left( \frac{\sqrt{ d }}{\mathrm{dist}(w, V^*)}+1 \right) \frac{\mathrm{dist}(w, V^*)^2}{2} \right\} \\
&\leq \exp \left( \tfrac12\mathrm{dist}(w, V^*)\sqrt{ d } \right)
\exp \left( \tfrac12\mathrm{dist}(w, V^*)\sqrt{ d } + \tfrac12\mathrm{dist}(w, V^*)^2 \right) \\
&= \exp \left( \mathrm{dist}(w, V^*)\sqrt{ d } + \tfrac{1}{2}\mathrm{dist}(w, V^*)^2 \right).
\end{align*}
Putting the above inequalities together, we finally obtain
\begin{align*}
\gamma(V+w) &\leq\gamma(V)\,\exp \left( \mathrm{dist}(w, V^*)\sqrt{ d } + \tfrac{1}{2}\mathrm{dist}(w, V^*)^2  - \tfrac{1}{2}\|w\|^2\right) \\
&= \gamma(V)\, \exp \left( \mathrm{dist}(w, V^*)\sqrt{ d } - \tfrac{1}{2} \|\proj_{V^*} w\|^2 \right),
\end{align*}
where the last equality uses the fact that $V^*$ is a closed convex cone (given any cone $V$) and the decomposition \citep[Section A.3.2]{hiriart-urruty_FundamentalsConvexAnalysis_2001}
\begin{equation*}
\|w\|^2 = \|\proj_{V^*}w\|^2 + \mathrm{dist}(w, V^*)^2.
\end{equation*}

The case $\mathrm{dist}(w, V^*)=0$ is straightforward, where for any $\kappa>1$, it holds that
\begin{equation*}
\gamma(V+w) \leq \gamma(V)\,h(\kappa) \exp\left( - \tfrac{1}{2}\|w\|^2 \right) 
=\gamma(V)\,\kappa^{d/2}\,\exp\left( -\tfrac{1}{2}\|w\|^2 \right).
\end{equation*}
Taking a limit as $\kappa \downarrow 1$ immediately yields $\gamma(V+w) \leq \gamma(V)\exp\left( -\frac{1}{2}\|w\|^2 \right)$. We conclude the proof of the upper bound by noting that $w=\proj_{V^*}w$ whenever $\mathrm{dist}(w, V^*)=0$.

\textbf{Lower bound}.
The above proof can be adapted to obtain a lower bound by studying the infima of the ratios $\phi(x+w) /\phi^\kappa(x)$ for $\kappa < 1$. We instead use a simpler method partially adapted from \citet[Section 3.3]{ledoux_ProbabilityBanachSpaces_1991}.

Because the case $\gamma(V)=0$ is trivial, assume $\gamma(V)>0$. We start with the integral
\begin{align*}
\gamma(V+w) &= \int_{V}\phi(x+w)dx \\
&= (2\pi)^{-d/2} \int_{V} e^{ - \|x+w\|^2/2}dx \\
&= (2\pi)^{-d/2}\,e^{-\|w\|^2/2} \, \int_{V} e^{ -w\cdot x }\, e^{ - \|x\|^2/2 } dx \\
&= e^{ -\|w\|^2/2 } \, \int_{V} e^{ -w\cdot x } \phi(x)dx 
= e^{ -\|w\|^2/2 } \gamma(V)\, \int_{\mathbb{R}^d} e^{ -w\cdot x }\,\, \indicator_{V}(x) \frac{\phi(x)}{\gamma(V)}dx,
\end{align*}
where $\indicator_{V}(x): \mathbb{R}^d \to \{ 0,1 \}$ is the Bernoulli indicator of $V$. Consider a random vector $X$ distributed according to the probability density $\indicator_{V}(\cdot)\,\phi(\cdot) /\gamma(V)$.
Rewriting the integrals using expectations, and using Jensen's inequality gives
\begin{equation*}
\int_{\mathbb{R}^d} e^{ -w\cdot x }\,\,\indicator_{V}(x) \frac{\phi(x)}{\gamma(V)}dx = \mathbb{E}\,e^{ -w\cdot X } \geq \exp \left( - \mathbb{E}\, w\cdot X \right) .
\end{equation*}
Noting that the random vector $X \in V$ (with probability 1), we may use the Moreau decomposition~\cite[Theorem A.3.2.5]{hiriart-urruty_FundamentalsConvexAnalysis_2001} to bound
\begin{align*}
w\cdot X &= \underbrace{ (\proj_{V^\circ}w)\cdot X }_{ \leq 0 } + (\proj_{V^{\circ\circ}}w)\cdot X \\
&\leq \|\proj_{V^{\circ\circ}}w\|\|X\| = \mathrm{dist}(w, V^\circ)\|X\| = \mathrm{dist}(-w, V^*)\|X\|,
\end{align*}
where $V^\circ=-V^*$ is the polar cone. Consequently,
\begin{equation*}
\exp \left( - \mathbb{E}\, w\cdot X \right) \geq \exp \left( - \mathrm{dist}(-w, V^*)\,\mathbb{E}\|X\| \right) .
\end{equation*}

What remains is to show that $\mathbb{E}\|X\|\leq \sqrt{ d }$ via Gaussian computations:
\begin{equation*}
\mathbb{E}\|X\| = \int_{\mathbb{R}^d} \|x\|\, \indicator_{V}(x) \frac{\phi(x)}{\gamma(V)}dx 
= \frac{1}{\gamma(V)}\, \mathbb{E}\,\|g\|\,\indicator_{V}(g) ,
\end{equation*}
where $g \sim \mathcal{N}(0,I_{d})$. Noting that the random variables $\rho:= \|g\|$ and $\theta := g/\|g\|$ are independent by properties of Gaussians, and that $\indicator_{V}(g)=\indicator_{V}(\theta)$ (with probability 1), we conclude that
\begin{align*}
\mathbb{E}\,\|g\|\,\indicator_{V}(g)
= \mathbb{E}\,[\rho\,\indicator_{V}(\theta)]
=\mathbb{E}[\rho]\,\mathbb{E}\,\indicator_{V}(\theta)
= \mathbb{E} \|g\|\,\mathbb{E}\,\indicator_{V}(g)
\leq \sqrt{ d }\,\gamma(V).
\end{align*}
Collecting everything yields
\begin{align*}
\gamma(V+w) &\geq e^{ -\|w\|^2/2 }\gamma(V) \exp \left( - \mathrm{dist}(-w, V^*)\,\mathbb{E}\|X\| \right)  \\
&\geq \gamma(V) \exp\left( - \tfrac{1}{2}\|w\|^2 - \mathrm{dist}(-w, V^*)\sqrt{d} \right),
\end{align*}
as required.
\end{proof}
\begin{proof}[\hypertarget{loc:gaussian_measure_of_the_margin_resulting_from_an_inward_shift.proof}Proof of \Cref{loc:gaussian_measure_of_the_margin_resulting_from_an_inward_shift.statement}]

To prove the lower bound, note that $T \cap [T+ \partial (\varepsilon \psi)(z)] \supseteq T \cap[T+w]$ for any subgradient $w \in \partial (\varepsilon\psi)(z) \cap T$.  Because the normal cone $T$ is convex and $w \in T$, we have the containment $T+ w \subseteq T+T=T$, and consequently, $T \cap[T+w]=T+w$. Putting these together and applying the lower bound from \Cref{loc:gaussian_measure_of_shifted_cones.statement} gives
\begin{align*}
\gamma(T \cap [T+ \partial (\varepsilon\psi)(z)]) &\geq \gamma(T \cap[T+w]) \\
&= \gamma(T+w) \\
&\geq \gamma(T)\exp\left( - \tfrac{1}{2}\|w\|^2 - \mathrm{dist}(-w, T^*)\sqrt{n} \right) \\
&= \gamma(T)\exp\left( - \tfrac{1}{2}\|w\|^2 - \|w\|\sqrt{n} \right),
\end{align*}
where the last inequality follows from the fact that $\mathrm{dist}(-w, T^*)=\mathrm{dist}(w, T^\circ)=\|\proj_{T}\,w\| = \|w\|$ for $w \in T$. 

The upper bound follows from the upper bound in \Cref{loc:gaussian_measure_of_shifted_cones.statement};
\begin{align*}
\gamma(T \cap [T+\varepsilon \nabla \psi(z)]) &\le \gamma(T+\varepsilon \nabla \psi(z)) \\
&\leq \gamma(T) \exp \left( - \tfrac{1}{2} \|\proj_{T^*} \,\varepsilon \nabla \psi(z)\|^2 + \mathrm{dist}(\varepsilon \nabla \psi(z), T^*)\sqrt{ n } \right).
\end{align*}
\end{proof}

\begin{proof}[\hypertarget{loc:corollary_of_the_inward_shift_result.proof}Proof of \Cref{loc:corollary_of_the_inward_shift_result.statement}]

For any $z \in \vrt(Q)$ and corresponding normal cone $T =N(z)$, it holds by \Cref{loc:gaussian_measure_of_the_margin_resulting_from_an_inward_shift.statement} that
\begin{align*}
\gamma(T \cap [T + \partial (\varepsilon\psi)(z)]) &\geq \gamma(T)\exp\left( -\tfrac{1}{2} \|w\|^2 - \|w\|\sqrt{ n } \right)
\end{align*}
for any $w \in\partial(\varepsilon\psi)(z) \cap T$. If $\varepsilon >0$, then by $\partial(\varepsilon \psi)=\varepsilon \partial \psi$, it holds that $\varepsilon v_{z} \in \partial(\varepsilon \psi)(z)\cap T$. If $\varepsilon=0$ on the other hand, $\varepsilon v_{z}=0 \in N_{\mathrm{dom}(\psi)}(z)\cap T = \partial(\varepsilon \psi)(z)\cap T$ trivially. Therefore, taking $w=\varepsilon v_{z}$, we have
\begin{align*}
\gamma(T \cap [T + \partial (\varepsilon\psi)(z)])  &\geq \gamma(T)\exp\left( - \tfrac{1}{2}\varepsilon^2\|v_{z}\|^2 - \varepsilon \|v_{z}\|\sqrt{n} \right) \\
&\geq \gamma(T)\exp \left( - \tfrac{1}{2}\varepsilon^2 B^2 - \varepsilon B \sqrt{ n } \right).
\end{align*}

Summing over all $z \in \vrt(Q)$ and using \Cref{loc:geometric_probabilistic_results.statement} gives
\begin{align*}
\mathbb{P}[\mathrm{ER}(\varepsilon)] &= \sum_{z \in \vrt(Q)} \gamma(N(z) \cap [N(z) +\partial (\varepsilon\psi)(z)]) \\
&\geq \sum_{z \in \vrt(Q)} \gamma(N(z))\exp \left( - \tfrac{1}{2}\varepsilon^2 B^2 - \varepsilon B \sqrt{ n } \right)
= \exp \left( - \tfrac{1}{2}\varepsilon^2 B^2 - \varepsilon B \sqrt{ n } \right),
\end{align*}
where the last equality uses the fact that $\sum_{z \in \vrt(Q)}\gamma(N(z))=\gamma(\mathbb{R}^n)=1$ for a (bounded) polytope $Q$. Taking the complement of $\mathrm{ER}(\varepsilon)$ and using the elementary inequality $1-t \leq e^{ -t }$ ($t \in \mathbb{R}$) concludes the first part of the proof.

To bound $\mathbb{E}\,\overline{\varepsilon}$, use~\eqref{eq:expectation_bound} together with the inequalities above to obtain 
\begin{align*}
\mathbb{E}\,\overline{\varepsilon} &= \int_{0}^\infty \mathbb{P}(\mathrm{ER}(\varepsilon))d\varepsilon \\
&\geq \int_{0}^\infty \exp \left( -\tfrac{1}{2}\varepsilon^2B^2 - \varepsilon B\sqrt{ n } \right) d\varepsilon \\
&\geq \int_{0}^{2\sqrt{ n }/B} \exp(-2\varepsilon B\sqrt{ n })d\varepsilon
= \frac{1-e^{ -4n }}{2B\sqrt{ n }}.
\end{align*}
\end{proof}

\begin{proof}[\hypertarget{loc:basic_exact_regularisation_theorem.proof}Proof of \Cref{loc:basic_exact_regularisation_theorem.statement}]

From \Cref{loc:geometric_probabilistic_results.statement} and the lower bound in \Cref{loc:gaussian_measure_of_the_margin_resulting_from_an_inward_shift.statement},
\begin{align*}
\prob(\mathrm{ER}(\varepsilon)) &= \textstyle\sum_{z  \in \vrt(Q)} \gamma(T \cap [T+ \varepsilon \nabla\psi(z)])\\
&\ge \textstyle\sum_{z  \in \vrt(Q)} \gamma(N(z))\exp\left( - \tfrac{1}{2} \varepsilon^2\|\nabla \psi(z)\|^2 - \varepsilon\|\nabla \psi(z)\|\sqrt{n} \right)\\
&\ge 1 - \max_{z  \in \vrt(Q)}( \varepsilon^2\|\nabla \psi(z)\|^2 + \varepsilon\|\nabla \psi(z)\|\sqrt{n}).
\end{align*}
Above, the second inequality utilizes the fact that $\sum_{z \in \vrt(Q)} \gamma(N(z))$ is a convex combination, as is made clear by \Cref{loc:normal_cones_and_faces.statement}.
So if we assume that
\begin{equation*}
\varepsilon \leq \tfrac{1}{2}\delta(\sqrt{n}\max_{z \in \vrt(Q)}\|\nabla \psi(z)\|)^{-1},
\end{equation*}
then
$\prob(\mathrm{ER}(\varepsilon)) \ge 1-\delta$, so the first part of the
result follows.

For the second part, since $\nabla \psi(z)  \in N^\star(z)$, then from \Cref{loc:gaussian_measure_of_the_margin_resulting_from_an_inward_shift.statement},
\begin{align*}
\prob(\mathrm{ER}(\varepsilon)) & \le  \textstyle\sum_{z \in \vrt(Q)} \gamma(N(z))
\exp\left(- \tfrac12 \varepsilon^2\|\nabla \psi(z)\|^2\right)\\
 &\le \exp\left(- \tfrac{1}{2} \varepsilon^2 \textstyle\min_{z \in \vrt(Q)}\|\nabla \psi(z)\|^2\right),
\end{align*}
and the second part of the result follows.
\end{proof}

\begin{proof}[\hypertarget{loc:representer_vectors_for_the_margin.proof}Proof of \Cref{loc:representer_vectors_for_the_margin.statement}]

We start by proving that the margin can be expressed as
\begin{equation*}
M(T,w)= \bigcup_{i=1}^{\ell} \,\{ x \in T \mid 0 \leq s_{i}\cdot x < s_{i}\cdot w \}.
\end{equation*}
For any $x \in M(T,w)$, it holds that $x \in T$ and $x-w \notin T$. As a result, $s_{i}\cdot x \geq 0$ for all $i \in [\ell]$, and there exists some $i_{0} \in [\ell]$ such that $s_{i_{0}}\cdot (x-w) < 0$. Consequently, $0 \leq s_{i_{0}}\cdot x < s_{i_{0}}\cdot w$ must hold. Conversely, suppose $0 \leq s_{i_{0}}\cdot x < s_{i_{0}}\cdot w$ holds for some $i_{0} \in [\ell]$. Then, $s_{i_{0}}\cdot(x-w) < 0$ and consequently $x \in M(T,w)$.

Next, notice that one can also write
\begin{equation*}
\{ x \in T \mid 0 \leq s_{i}\cdot x < s_{i}\cdot w \} = \{ x \in T \mid 0 \leq s_{i}\cdot x < (s_{i}\cdot w)_{+} \}
\end{equation*}
(if $s_{i}\cdot w <0$ then both sides equal the empty set). Using this fact, and noting that $s_{i}\cdot \tilde{w} = (s_{i}\cdot w)_{+}$ for all $i \in [\ell]$ by the definition of a representer vector, we conclude the proof as
\begin{align*}
M(T,w) &= \textstyle\bigcup_{i=1}^\ell \{ x \in T \mid 0 \leq s_{i}\cdot x < (s_{i}\cdot w)_{+} \} \\
&= \textstyle\bigcup_{i=1}^\ell \{ x \in T \mid 0 \leq s_{i}\cdot x < s_{i}\cdot \tilde{w} \} = M(T, \tilde{w}).
\end{align*}
\end{proof}

\begin{proof}[\hypertarget{loc:bound_via_representer_vectors.proof}Proof of \Cref{loc:bound_via_representer_vectors.statement}]

We use the key properties of the representer vectors to reduce to the setting of \Cref{loc:gaussian_measure_of_the_margin_resulting_from_an_inward_shift.statement}.

\textbf{Lower bound}. Simplifying as before and passing through the definition of the margin gives
\begin{align*}
\gamma(T \cap [T+ \partial(\varepsilon \psi)(z)]) &\geq \gamma(T \cap[T+w]) \\
&= \gamma(T) - \gamma(M(T, w)) \\
&= \gamma(T) - \gamma(M(T, \tilde{w}))
= \gamma(T \cap[T+\tilde{w}]).
\end{align*}
Recall that the representer vector $\tilde{w}= S_{T}^{-1}(S_{T}w)_{+}$ not only induces the same margin as $w$, but is also a member of $T$. Therefore, using convexity of the cone $T$ and applying the lower bound from \Cref{loc:gaussian_measure_of_shifted_cones.statement} gives
\begin{align*}
\gamma(T \cap [T+\tilde{w}]) &= \gamma(T+\tilde{w}) \\
&\geq \gamma(T)\exp\left( - \tfrac{1}{2}\|\tilde{w}\|^2 - \mathrm{dist}(-\tilde{w}, T^*)\sqrt{n} \right) \\
&= \gamma(T)\exp\left( - \tfrac{1}{2}\|\tilde{w}\|^2 - \|\tilde{w}\|\sqrt{n} \right).
\end{align*}

\textbf{Upper bound}. Following an identical calculation, we have
\begin{equation*}
\gamma(T \cap[T+\varepsilon \nabla \psi(z)]) = \gamma(T \cap[T+\tilde{w}]) = \gamma(T+\tilde{w})
\end{equation*}
where now $\tilde{w}:= \varepsilon\, S_{T}^{-1}(S_{T}\nabla \psi(z))_{+}$. Applying the upper bound from \Cref{loc:gaussian_measure_of_shifted_cones.statement} gives
\begin{equation*}
\gamma(T+\tilde{w}) \leq \gamma(T) \exp \left( - \tfrac{1}{2} \|\proj_{T^*} \tilde{w}\|^2 + \mathrm{dist}(\tilde{w}, T^*)\sqrt{ n } \right).
\end{equation*}
\end{proof}

\begin{proof}[\hypertarget{loc:partition_of_the_margin.proof}Proof of \Cref{loc:partition_of_the_margin.statement}]

For any $x \in M$, $x - w \notin M$, so we may define
\begin{equation*}
t := \sup \{ q  \in [0,1) \mid (x - qw) \in T\}.
\end{equation*}
Note that the supremum is attained, because this supremum can be recast as the supremum of a continuous function over a compact set. The attainment implies that $t \in [0,1)$.
Recall that $T =  \{x \in \mathbb{R}^n \mid \forall i \in [\ell], \, s^i \cdot x \ge 0\}$. There is at least one index $i \in [\ell]$ for which $s^i \cdot (x- t w) = 0$ and $\forall \varepsilon>0,$ $s^i \cdot (x- (t + \varepsilon) w) < 0$. Fixing $i$ to be any such index, 
we next show that $\forall x  \in M,$ $x \in P_i$.

Note that $x - tw \in T^i$ by definition of $t$. With $z := x-tw$,  we can write $x =z + tw$ for some $z \in T^i$, $t  \in [0,1)$. Second, we show $s^i \cdot w > 0$. Recall that  $s^i \cdot (x - w) < 0$ by definition of $t$, because $1> t$. Also, $s^i \cdot (x - tw)= 0.$
Adding these equations together so as to cancel the term that includes $x$, and then using the fact that $t < 1$, we find that $s^i \cdot w > 0$.

We have thus shown that $M \subseteq \bigcup_{i=1}^\ell P^i$. Because of this,
\begin{equation}
\label{eq:main:in:proof}
\mu(M)  \le \textstyle\mu\left(\bigcup_{i=1}^\ell P^i\right) \le \sum_{i = 1}^\ell \mu(P^i),
\end{equation}
so we have shown the first part of the result.

In the remainder of this proof, we assume that $w \in T$. In this setting, we show that $\bigcup_{i=1}^\ell P^i  \subseteq M(T, w)$. Take any $x \in P^i$ for any $i \in [\ell]$. From the definition of $P^i$, $x = (z + tw)$ for some $z \in T^i$ and $t \in [0,1)$. But then $s^i\cdot(x-w) = s^i \cdot (z + (t-1)w) < 0$, and so $x - w  \notin T$. This argument holds because $w  \in T$ implies that $\forall j  \in [\ell], w  \cdot s_j  \ge 0$, and $z  \in T^i  \implies s^i \cdot z = 0$. Additionally, since both $w, z \in T$,  we find that $x = (z + tw) \in T$ because convex cones are closed under sums and and positive scalar multiplication. Hence $x  \in T \setminus [T+w] =: M(T, w)$. This shows the reverse inclusion.

It only remains to show that $P^i  \cap P^j$ is of measure zero for any $i  \neq  j$,
because then the second inequality in \eqref{eq:main:in:proof} holds with
equality, and we are done. By \citet[Proposition 1(b)]{lu_NormalFansPolyhedral_2008},
the set $T^i \cap T^j$, being an intersection of distinct facets of a
single cone $T$, has a dimension
$n-2$ (in the sense of the dimension of the affine hull). 
Then $P^i \cap P^j  = \bigcup_{t  \in [0, 1)} [T^i  \cap T^j] + tw$, which is a set of dimension $n-1$,
and it therefore has zero Gaussian measure.
\end{proof}

\begin{proof}[\hypertarget{loc:gaussian_measure_of_the_margin.proof}Proof of \Cref{loc:gaussian_measure_of_the_margin.statement}]

Thanks to \Cref{loc:partition_of_the_margin.statement}, we only need to
evaluate $\gamma(P^i)$ for some $i \in [\ell]$.
Fix $i  \in [\ell]$, and let $w_\perp := (w  \cdot s^i)_+s^i$, and $w_\parallel  =  w- w_\perp$. We find $\gamma(P^i)$ with an iterated integral along $T^i$ and $w$.
\begin{subequations}
\begin{align}
\gamma(P^i) &= \frac{1}{(2\pi)^{n/2}}\int_{t \in [0, 1]} \int_{x \in T^i} \exp\left(- \frac{\|x + tw\|^2}{2}\right)dx (w \cdot s^i)_+dt \label{eq:equality_first:1}\\
&=\frac{(w \cdot s^i)_+}{(2\pi)^{n/2}}\int_{t \in [0, 1]} \int_{x \in T^i} \exp\left(- \frac{\|x + tw_{\parallel }\|^2 + t^2\|w_{\perp}\|^2}{2}\right)dx dt \label{eq:equality_first:2}\\
&= \frac{(w \cdot s^i)_+}{\sqrt{2\pi}}\int_{[0, 1]} \gamma_{\mathrm{rel}}(T^i+ tw_{\parallel })\exp\left(-t^2 \frac{(w \cdot s^i)_+^2}{2}\right)dt. \label{eq:equality_first:3}
\end{align}
\end{subequations}
In \Cref{eq:equality_first:2} we decompose the vector $x+tw$ into two
orthogonal components, and in \Cref{eq:equality_first:3}, we use the fact that
$\|w_\perp\| = (w  \cdot s^i)$, which we can substitute by $(w \cdot s^i)_+$
without changing the value of the expression, because of the factor of $(w \cdot s^i)_+$ in front of the integral which makes the entire expression zero when $w \cdot s^i < 0$.

\textbf{Upper bound}. Note that $\exp\left(-t^2 \frac{(w \cdot s^i)_+^2}{2}\right)  \le 1$, and with
an upper bound on $\gamma_{\mathrm{rel}}(T^i+ tw_{\parallel })$ from
\Cref{loc:simpler_measure_of_shifted_cones.statement},
\begin{align*}
\gamma(P_i) &\le \gamma_{\mathrm{rel}}(T^i)\frac{ (w \cdot s^i)_+}{\sqrt{2\pi}}\int_{[0, 1]} \exp \left( t\sqrt{n-1} \|w_\parallel \|\right)dt\\
 & \le  \gamma_{\mathrm{rel}}(T^i)\frac{ (w \cdot s^i)_+}{\sqrt{2\pi}}\exp \left( \sqrt{n-1} \|w_\parallel \|\right),
\end{align*}
which we get by setting $t = 1$ inside the integral, giving us a uniform upper-bound
on the integrand. The result then follows from \Cref{loc:partition_of_the_margin.statement} with $\|w_\parallel \| \le \|w\|$.

\textbf{Lower bound}.
In what follows, we lower-bound \Cref{eq:equality_first:3} by first applying \Cref{loc:simpler_measure_of_shifted_cones.statement} and then fixing $t = 1$ to find a uniform lower-bound on the
integrand.
\begin{align*}
\gamma(P_i) & \ge \frac{(w \cdot s^i)_+}{\sqrt{2\pi}}\int_{[0, 1]} \gamma_{\mathrm{rel}}(T^i+ tw_{\parallel })\exp\left(-\tfrac12 t^2 (w \cdot s^i)_+^2\right)dt\\
&\ge \gamma_{\mathrm{rel}}(T^i)\frac{ (w \cdot s^i)_+}{\sqrt{2\pi}} \int_{[0, 1]} \exp\left(- t\sqrt{n-1} \|w_\parallel\| - \tfrac12t^2\|w_{\parallel }\|^2\right)\exp\left(-\tfrac12 t^2(w \cdot s^i)_+^2\right) \\
&\ge \gamma_{\mathrm{rel}}(T^i)\frac{ (w \cdot s^i)_+}{\sqrt{2\pi}}  \exp\left(- \sqrt{n-1} \|w_\parallel\| - \tfrac12\|w_{\parallel }\|^2\right)\exp\left(-\tfrac12 (w \cdot s^i)_+^2\right) \\
&\ge \gamma_{\mathrm{rel}}(T^i)\frac{(w \cdot s^i)_+}{\sqrt{2\pi}} \exp \left(- \sqrt{n-1}\|w \| - \tfrac12\|w\|^2\right).
\end{align*}
Then because $w \in T$, the lower bound follows from \Cref{loc:partition_of_the_margin.statement}.
\end{proof}
\begin{proof}[\hypertarget{loc:applied_margin_measure_for_exact_regularization.proof}Proof of \Cref{loc:applied_margin_measure_for_exact_regularization.statement}]

For the upper bound, note that $N(z) \setminus [N(z)+ \partial (\varepsilon \psi)(z)] \subseteq N(z) \setminus [N(z)+w]$ for any subgradient $w \in \partial (\varepsilon\psi)(z)$. With this, the upper bound and the first lower bound follow from 
\Cref{loc:gaussian_measure_of_the_margin.statement} with $T = N(z)$ and $w \in \partial(\varepsilon \psi)(z)$,
by identifying $\varepsilon v = w$.

For the second lower-bound, let $w  = \varepsilon\nabla \psi(z)$ and 
$\tilde{w}  = S^{-1}(S w)_+$. As shown in \Cref{loc:representer_vectors_for_the_margin.statement}, the margin associated with
$\tilde{w}$ is exactly the same as the margin associated with $w$. 
We use the lower bound in 
\Cref{loc:gaussian_measure_of_the_margin.statement}
with $\tilde{w}$ instead of $w$. Finally, since $F(N(z), w)$ only depends on $(s^i  \cdot w)_+$,
and $(s^i  \cdot \tilde{w})_+  = (s^i  \cdot w)_+$, then
$F(N(z), \tilde{w}) = F(N(z), w)$,
so we may write $F(N(z), w)$ in place of $F(N(z), \tilde{w})$. With $\varepsilon v=w$, the result follows.
\end{proof}
\begin{proof}[\hypertarget{loc:tightness_of_the_loo_ball_case.proof}Proof of \Cref{loc:tightness_of_the_loo_ball_case.statement}]

A necessary condition for optimality for this problem is that $\sgn(\nabla f(z^\star)) \approx z^\star$, with $(\approx)$ meaning that the two sides do not have opposite signs ($0 \approx 1$ but $-1 \not \approx 1$). Therefore, if exact regularization is to succeed,
there must be a vertex $z^\star$ of $\Binf$ such that $\sgn(g) \approx z^\star$ (optimality of $P_0$), and such that $\sgn(g - \varepsilon  z^\star) \approx z^\star$ (optimality of $P_\varepsilon$).

Since the sign of $-\varepsilon z^\star$ is elementwise opposite to that of $g$, the condition $|g| - \varepsilon \allone  \ge 0$, 
is necessary, where $\allone \in \mathbb{R}^n$ is the all-ones vector.

We want to show that exact regularization fails with high probability, so we bound the probability that it succeeds: 
\begin{align}
\prob(\mathrm{ER}(\varepsilon)) &\le \prob(\min(|g|) \ge \varepsilon)
=  \prob\left(|g_1| \ge  \varepsilon\right)^n 
=\left(1 - \pr\left(|g_{1}| < \varepsilon\right)\right)^n. \label{eq:reversed:3}
\end{align}
When $\varepsilon \le 1$, we get $\pr\left(|g_{1}| < \varepsilon\right)  \ge \sqrt{2/(\pi e)}\cdot \varepsilon$.
This lower bound arises from a constant lower-bound (with bounded support) of the Gaussian density function.
With the identity $1+x  \le e^x$, it follows from \Cref{eq:reversed:3} that
\begin{equation}
\label{eq:first:bound}
\prob(\textrm{ER}(\varepsilon)) \le  \exp\left(-\sqrt{2/(\pi e)}\cdot n \varepsilon \right).
\end{equation}
Now consider the case $\varepsilon  \ge 1$. By \citet[Proposition 2.1.2]{vershynin_HighDimensionalProbabilityIntroduction_2018},
\begin{align*}
\prob(|g_1| &\ge  \varepsilon)^n  \le \left(\tfrac{2}{\pi}\right)^{n/2} \exp\left(-\tfrac12 n \varepsilon^2\right)
\le \exp\left(-\tfrac12 n \varepsilon^2\right)
\le \exp\left( -\tfrac12 n \varepsilon \right).
\end{align*}
This is a larger upper bound than that of \eqref{eq:first:bound},
and so the first part of the result follows. For the second part, it suffices to consider the elementary inequality
on $x  \in [-1, 0]$, that $\frac{1}{2}x + 1 \geq e^x$.
\end{proof}

\section{Discussion}
\label{loc:body.discussion}

We analyzed exact regularization in the random linear program $P_{0}$ with a standard Gaussian cost vector, deriving expressions for the probability of $\mathrm{ER}(\varepsilon)$ in terms of normal cone geometry. The polyhedral inner cones and margins play a central role in the analysis. We provided two-sided estimates of their Gaussian measure, which yield simple, explicit expressions in interesting special cases, including quadratically regularized linear programs on the Birkhoff polytope. These estimates agree with numerical experiments.

When exact regularization fails, deterministic bounds exist for $\mathrm{dist}(x_{\varepsilon}, \mathrm{Sol}(P_{0}))$, but these bounds involve parameters that are difficult to evaluate~\cite{friedlander_ExactRegularizationConvex_2008}. The Gaussian randomization framework we develop may provide tools needed to derive high-probability and expectation bounds on solution distance when exact regularization fails. The normal-cone techniques and Gaussian measure estimates may enable characterization of the distributions of the regularized solutions $x_{\varepsilon}$, conditioned on failure events.

Recent work reveals that support monotonicity in quadratically constrained regularized LPs fails in general, but tends to hold in experiments~\cite{gonzalez-sanz_MonotonicityQuadraticallyRegularized_2025}. These empirical observations suggest an underlying probabilistic structure. Can we quantify the probability of occurrence for randomized cost vectors?

The polyhedral structure of LPs allows explicit representations of normal cones and their Gaussian measure. For more general conic problems that include quadratic and semidefinite constraints, second-order cones and linear matrix inequalities characterize the corresponding normal cone geometry. Our main technical result on the Gaussian measure of shifted cones (\cref{loc:gaussian_measure_of_shifted_cones.statement}) does not depend on conic polyhedrality and thus generalizes to those settings. However, the analogs of polyhedral inner cones and margins present technical challenges.

\section{Acknowledgments}

The authors thank Andrew Warren for highlighting connections to quadratically regularized optimization transport and to Yifan Sun for fruitful discussions.

\bibliographystyle{abbrvnat}
\bibliography{bibliography}

\appendix
\section{Auxiliary lemma}
\label{sec:aux_proofs}

\begin{lemma}[Intersections of normal cones]
\label{loc:auxiliary_result_intersections_of_normal_cones.statement}
Let $Q$ be a closed convex set in $\mathbb{R}^n$, and let $z_{1}, z_{2} \in Q$. Then, $N\left( \tfrac12(z_{1}+z_{2}) \right) = N(z_{1}) \cap N(z_{2})$.
\end{lemma}
\begin{proof}[\hypertarget{loc:auxiliary_result_intersections_of_normal_cones.proof}Proof of \Cref{loc:auxiliary_result_intersections_of_normal_cones.statement}]

($\subseteq$) Let $F_{Q}(v) := \{ x \in Q \mid v\cdot x = \max_{y \in Q} v\cdot y \}$
be the face of $Q$ exposed by $v$ \cite[Proposition A.5.3.3]{hiriart-urruty_FundamentalsConvexAnalysis_2001}. For any $x\in Q$, we have $v\in N(x)$ if and only if $x\in F(v)$. Thus, if $v \in N\big( (z_{1}+z_{2})/2 \big)$, then $\tfrac{1}{2}(z_{1}+z_{2}) \in F(v)$. By the defining face property, this implies $z_{1},z_{2} \in F(v)$, i.e., $v\in N(z_{1})\cap N(z_{2})$.

($\supseteq$) If $v \in N(z_{1})\cap N(z_{2})$, then $z_{1},z_{2} \in F(v)$. Since $F(v)$ is a face, $\tfrac{1}{2}(z_{1}+z_{2}) \in F(v)$, hence $v \in N\big( (z_{1}+z_{2})/2 \big)$.
\end{proof}

\end{document}